\documentclass[11pt,english]{article}

\usepackage[margin=1.83 cm,bottom=19mm,footskip=7mm,top=19mm]{geometry}

\usepackage{amsthm}
\usepackage{amsmath}
\usepackage{amssymb}
\usepackage{setspace}
\usepackage{mathtools}
\usepackage{graphicx}
\usepackage[hidelinks]{hyperref}
\usepackage{thm-restate}
\usepackage{cleveref}
\usepackage{enumitem}
\usepackage{framed}
\usepackage{caption}
\usepackage[hang,flushmargin]{footmisc}

\usepackage{tabu}
 
\usepackage[square,sort,comma,numbers]{natbib}
\usepackage{tikz}
\usepackage{mathdots}
\usepackage{xcolor}
\usetikzlibrary{calc}
\usetikzlibrary{decorations.pathreplacing}
\usetikzlibrary{positioning,patterns}
\usetikzlibrary{arrows,shapes,positioning}
\usetikzlibrary{decorations.markings}

\def \smvx {circle[radius = .07][fill = black]}

\tikzstyle{edge}=[very thick]
\definecolor{bostonuniversityred}{rgb}{0.8, 0.0, 0.0}
\definecolor{arsenic}{rgb}{0.23, 0.27, 0.29}
\tikzstyle{diredge}=[postaction={decorate,decoration={markings,
		mark=at position .65 with {\arrow[scale = 1]{stealth};}}}]

\tikzstyle{diredge2}=[postaction={decorate,decoration={markings,
		mark=at position .55 with {\arrow[scale = 1]{stealth};}}}]

\newcommand{\defPt}[3]{
	\def \pt {(#1, #2)}
	\coordinate [at = \pt, name = #3];
}

\newcommand{\defPtm}[2]{

	\coordinate [at = #1, name = #2];
}

\tikzset{
   K5/.pic={
     \foreach \x in {1,...,5}{%
    \pgfmathparse{(\x-1)*360/5}
    \node[draw,circle,fill=black, inner sep=1 pt] (N-\x) at (\pgfmathresult:1 cm) [thick] {};
  }
  
  \foreach \x in {1,...,4}{%
    \foreach \y in {\x,...,5}{%
        \path (N-\x) edge[dotted,line width=1 pt] (N-\y);
  }
  }
  }
}

\tikzset{
   K4/.pic={
     \foreach \x in {1,...,4}{%
    \pgfmathparse{(\x-1)*360/4}
    \node[draw,circle,fill=black, inner sep=1 pt] (N-\x) at (\pgfmathresult:1 cm) [thick] {};
  }
  
  \foreach \x in {1,...,3}{%
    \foreach \y in {\x,...,4}{%
        \path (N-\x) edge[dotted,line width=1 pt] (N-\y);
  }
  }
  }
}

\newcommand{\fitellipsis}[3] 
{\draw [fill=white] let \p1=(#1), \p2=(#2), \n1={atan2(\y2-\y1,\x2-\x1)}, \n2={veclen(\y2-\y1,\x2-\x1)}
    in ($ (\p1)!0.5!(\p2) $) ellipse [ x radius=\n2/2+0.1cm, y radius=#3cm, rotate=\n1];
}

\newcommand{\fitellipsisnfill}[3] 
{\draw [] let \p1=(#1), \p2=(#2), \n1={atan2(\y2-\y1,\x2-\x1)}, \n2={veclen(\y2-\y1,\x2-\x1)}
    in ($ (\p1)!0.5!(\p2) $) ellipse [ x radius=\n2/2+0.1cm, y radius=#3cm, rotate=\n1];
}

\theoremstyle{plain}

\newtheorem*{thm*}{Theorem}
\newtheorem{thm}{Theorem}[section]
\Crefname{thm}{Theorem}{Theorems}

\newtheorem*{lem*}{Lemma}
\newtheorem{lem}[thm]{Lemma}
\Crefname{lem}{Lemma}{Lemmas}

\newcounter{cclaim}
\newtheorem*{claim*}{Claim}

\crefname{claim}{Claim}{Claims}
\Crefname{claim}{Claim}{Claims}

\newtheorem{prop}[thm]{Proposition}
\Crefname{prop}{Proposition}{Propositions}

\newtheorem{cor}[thm]{Corollary}
\crefname{cor}{Corollary}{Corollaries}

\crefname{conj}{Conjecture}{Conjectures}

\newtheorem{qn}[thm]{Question}
\Crefname{qn}{Question}{Questions}

\Crefname{obs}{Observation}{Observations}

\Crefname{ex}{Example}{Examples}

\theoremstyle{definition}

\Crefname{prob}{Problem}{Problems}

\Crefname{defn}{Definition}{Definitions}

\theoremstyle{remark}

\renewenvironment{proof}[1][]{\begin{trivlist}
\item[\hspace{\labelsep}{\bf\noindent Proof#1.\/}] }{\qed\end{trivlist}}

\newcommand{\remove}[1]{}

\newcommand{\ceil}[1]{
    \lceil #1 \rceil
}
\newcommand{\floor}[1]{
    \lfloor #1 \rfloor
}

\newcommand{\eps}{\varepsilon}
\renewcommand{\P}{\mathbb{P}}
\newcommand{\G}{\mathcal{G}}
\newcommand{\HH}{\mathcal{H}}

\newcommand{\property}[2]
{$#1$-local $#2$-independence property}

\title{\texorpdfstring{\vspace{-1.3cm}}{}
Large independent sets from local considerations}
\date{}
\author{
Matija Buci\'c\thanks{School of Mathematics, Institute for Advanced Study and Department of Mathematics, Princeton University, Princeton, USA. Email: \href{mailto:matija.bucic@ias.edu} {\nolinkurl{matija.bucic@ias.edu}}.}
\and
Benny Sudakov\thanks{Department of Mathematics, ETH, Z\"urich, Switzerland. Email:
\href{mailto:benjamin.sudakov@math.ethz.ch} {\nolinkurl{benjamin.sudakov@math.ethz.ch}}.
Research supported in part by SNSF grant 200021\_196965.}
}

\begin{document}

\maketitle

\begin{abstract}

    The following natural problem was raised independently by Erd\H{o}s-Hajnal and Linial-Rabinovich in the early '90s. How large must the independence number $\alpha(G)$ of a graph $G$ be whose every $m$ vertices contain an independent set of size $r$? In this paper, we discuss new methods to attack this problem. 
    
    The first new approach, based on bounding Ramsey numbers of certain graphs, allows us to improve the previously best lower bounds due to Linial-Rabinovich, Erd\H{o}s-Hajnal and Alon-Sudakov. As an example, we prove that any $n$-vertex graph $G$ having an independent set of size $3$ among every $7$ vertices has $\alpha(G) \ge \Omega(n^{5/12})$. This confirms a conjecture of Erd\H{o}s and Hajnal that $\alpha(G)$ should be at least $n^{1/3+\eps}$ and brings the exponent halfway to the best possible value of $1/2$.
    
    Our second approach deals with upper bounds. It relies on a reduction of the original question to the following natural extremal problem. What is the minimum possible value of the $2$-density\footnote{The \textit{$2$-density} of a graph $H$ is defined as $m_2(H):= \max\limits_{H' \subseteq H, |H'| \ge 3} \frac{e(H')-1}{|H'|-2}.$} of a graph on $m$ vertices having no independent set of size $r$? This allows us to improve previous upper bounds due to Linial-Rabinovich, Krivelevich and Kostochka-Jancey. 
    
    As part of our arguments, we link the problem of Erd\H{o}s-Hajnal and Linial-Rabinovich and our new extremal $2$-density problem to a number of other well-studied questions. This leads to many interesting directions for future research.
\end{abstract}

\section{Introduction}
In this paper, we study the following classical problem. If we know that any $m$ vertices of a graph contain an independent set of order $r$ how large can the independence number of the whole graph be? The study of this problem for a specific choice of parameters dates back almost $60$ years, with the first published result being due to Erd\H{o}s and Rogers \cite{erdos-rogers} in 1962.

Over the years this problem attracted a lot of attention. Originally the focus was on the instance of the problem in which we keep the sizes of independent sets we want to find locally and in the whole graph to be fixed and small. In other words, if we forbid in $G$ an independent set of size $s,$ how big a subset of vertices one can find without an independent set of size $r$? Choosing $r=2$ precisely recovers the usual Ramsey problem and was in fact the original motivation behind the general question. This question became known as the Erd\H{o}s-Rogers problem and has been extensively studied, for some examples see \cite{erdos-rogers,bollobas-hind,michael-q-s,wolfowitz,dudek, benny, benny-2,gowers-janzer} and a recent survey \cite{survey} due to Dudek and R\"odl.

In the early 90's Erd\H{o}s and Hajnal \cite{erdos-problems} and independently Linial and Rabinovich \cite{L-R} propose changing the perspective and fixing the local parameters $m$ and $r$ instead. In other words, asking what can be said about the independence number of the whole graph if we know that any small number of vertices $m$ contain an independent set of size $r$. This frames the problem squarely under the so-called local-global principle, stating that one can obtain global understanding of a structure from having a good understanding of its local properties, or vice versa. This phenomenon has been ubiquitous in many areas of mathematics and beyond, see e.g. \cite{local-global-1,local-global-2,local-global-4,gromov}. In fact, one can define an $m$-local independence number $\alpha_m(G)$ of a graph $G$ to be the minimum independence number we can find among subgraphs of $G$ on $m$ vertices and the problem becomes relating the local independence number to the independence number of $G$ itself, the ``global'' independence number. 
In particular, we are interested in the smallest possible size of $\alpha(G)$ in an $n$-vertex graph satisfying $\alpha_m(G) \ge r$. 

In this paper we discuss two new approaches for attacking this problem, which allow us to significantly improve previously best-known bounds due to Linial and Rabinovich \cite{L-R}, Erd\H{o}s-Hajnal \cite{erdos-problems}, Alon and Sudakov \cite{alon-sudakov}, Krivelevich \cite{michael-critical} and Kostochka and Jancey \cite{kost}. In the case of lower bounds, we improve their results for at least half of the possible choices of $m$ and $r$ and in the case of upper bounds for essentially all choices. Moreover, we believe that both approaches have the potential for further improvements. 

The initial approach of Linial and Rabinovich \cite{L-R} and independently Alon and Sudakov \cite{alon-sudakov} reduces the lower bound problem to the question of bounding from above Ramsey numbers of a clique of size $k=\ceil{\frac{m}{r-1}}$ vs a large independent set. Our new idea is that one can find other ``forbidden'' graphs whose Ramsey numbers perform better. For this to work we need to obtain upper bounds on the Ramsey numbers of our new graphs vs a large independent set, which often turns out to be an interesting problem in its own right. See the beginning of \Cref{sec:2.1} for a more detailed illustration of our new approach.

The above-introduced parameter $k$ controls in large part the known lower bounds for $\alpha(G)$ among all graphs satisfying $\alpha_m(G) \ge r$. Linial and Rabinovich \cite{L-R} determine the answer precisely if $k \le 2$, i.e.\ for $m \le 2r-2$. For $k=3$, they show that an $n$ vertex graph satisfying $\alpha_m(G) \ge r$ must have $\alpha(G) \ge n^{1-\frac{2}{r-1}-o(1)}$ if $m=2r-1$ and $\alpha(G) \ge \Omega(n^{1/2})$ for the rest of the range $m \le 3r-3$. Our first result improves the exponent in their bounds for the first half of this range. Moreover, the improvement in the exponent is by a constant factor independent of $r$, unless $m = 2r-1$. 

\begin{restatable}{prop}{kisthree}\label{prop:k=3-lower-bound}
Let $m=2r-2+t$ for $1 \le t \le r-1$. Then any $n$-vertex graph $G$ satisfying $\alpha_m(G) \ge r$ has $\alpha(G) \ge \Omega(n^{1-1/\ell}),$ where $\ell=\floor{\frac{r-1}{t}}+1$.
\end{restatable}

In the general case of $k \ge 4$, Linial and Rabinovich and independently Alon and Sudakov show that an $n$-vertex graph satisfying $\alpha_m(G) \ge r$ must have $\alpha(G) \ge \Omega(n^{\frac{1}{k-1}}).$  We improve the exponent in these bounds for the first half of the range for any $k$.

\begin{thm}\label{thm:main-m-3}
Let $k=\ceil{\frac{m}{r-1}}$ and let us assume $m \le (k-\frac{1}{2})(r-1).$ Then any $n$-vertex graph $G$ satisfying $\alpha_m(G) \ge r$ has  $\alpha(G) \ge \Omega(n^{\frac{1}{k-3/2}})$.
\end{thm}
Going beyond $k-2$ in the denominator of the exponent in the above theorem seems likely to require an improvement over the best-known upper bounds on Ramsey numbers, which have not seen an improvement in the exponent since the initial paper of Erd\H{o}s and Szekeres \cite{erdos-szekeres} from 1935. This means our result is in some sense halfway between the previously best bound and the Ramsey barrier. 

The key part of the above result is actually the special case of $r=3$. This is due to an easy observation which allows us to generalise any improvement in this case to the first half of the range as above, for any $r$. The first interesting instance here, which actually lead us to the general improvements above, is $m=7$ and $r=3$ in which case we can obtain an even better bound. Studying this case was explicitly proposed by Erd\H{o}s and Hajnal \cite{erdos-problems} who observed that any graph $G$ on $n$ vertices with $\alpha_7(G) \ge 3$ must have $\alpha(G) \ge \Omega(n^{1/3})$ and that such a graph $G$ exists with $\alpha(G) \le O(n^{1/2})$. They conjectured that neither of these bounds is tight. Our next result confirms their first conjecture.
\begin{thm}\label{thm:main-7-3}
Any $n$-vertex graph $G$ with $\alpha_7(G) \ge 3$ has $\alpha(G) \ge n^{5/12-o(1)}$.
\end{thm}
By the aforementioned observation, this actually gives the same improved bound for the first half of the range for any instance with $k=4,$ i.e.\ for $3r-2\le m \le 3.5(r-1)$.

To prove an upper bound on the minimum possible $\alpha(G)$ among all graphs with $\alpha_m(G)\ge r$, one needs to find a graph which has small independence number while having big independent sets spread around everywhere. Given the close relation of our problem to Ramsey numbers, random graphs are natural candidates for such examples. A binomial random graph $\G(n,p)$ is an $n$-vertex graph in which we include every possible edge independently with probability $p$. Understanding $\alpha_m(G)$ in $G \sim \G(n,p)$ turns out to be an interesting problem in its own right. Observe that the requirement $\alpha_m(G)\ge r$ may be rephrased as stating that $G$ contains no copy of an $m$-vertex graph $H$ with $\alpha(H)\le r-1$ as a subgraph.
A standard application of Lov\'asz local lemma tells us that if we are only forbidding a single graph $H$, then the largest $p$ we can take is controlled by the $2$-density of $H$ (see \Cref{sec:2-density} for the definition of $2$-density and more details). If we are instead forbidding a family of graphs the correct parameter turns out to be the minimum of the $2$-densities over all graphs in our family. This reduces our problem to the following natural extremal question, which we propose to study. What is the minimum value of the $2$-density of an $m$-vertex graph $H$ with $\alpha(H) \le r-1$? If we denote the answer to this question by $M(m,r)$ the above discussion leads us to the following reduction.

\begin{restatable}{prop}{equiv}\label{prop:m2-loc-ind}
Let $m,r$ be fixed, $m \ge 2r-1 \ge 3$ and $M=M(m,r)$. Then for any $n$ there exists an $n$-vertex graph $G$ with $\alpha_m(G)\ge r$ and $\alpha(G) \le n^{1/M+o(1)}$.
\end{restatable}

The value of $M(m,r)$, and hence also our upper bounds for the local to global independence number problem, are mostly controlled by the same parameter $k=\ceil{\frac{m}{r-1}}$ as before. Some intuition behind this, suggested by Linial and Rabinovich \cite{L-R}, is that a natural example of an $m$ vertex graph with independence number at most $r-1$ is a vertex disjoint union of $r-1$ cliques with sizes as equal as possible (in other words complement of a Tur\'an graph on $m$ vertices with no clique of size $r$). This graph clearly has no independent set of size $r$ and we picked the clique sizes as equal as possible in order to minimise the $2$-density. Our parameter $k$ is simply the size of a largest clique in this example. 

Turning to the results, we start once again with the range $2r-1\le m\le 3r-3$, i.e.\ $k=3$. Here Linial and Rabinovich show that there exist $n$-vertex graphs $G$ satisfying $\alpha_m(G) \ge r$ and $\alpha(G) \le n^{1-1/(8r-4)}.$ We improve the exponent in this bound for the whole range. Moreover, the improvement in the exponent is by a constant factor independent of $r$, towards the end of the range.
 
\begin{prop}\label{prop-alpha-k=3}
Let $m\ge 2r-1 \ge 3,$ for any $n$ there exists an $n$-vertex graph $G$ satisfying  $\alpha_m(G) \ge r$ with $\alpha(G) \le n^{1-\frac{1}{2r-2}+o(1)}$ and if $m\ge 3r-4>2$ with $\alpha(G) \le n^{\frac{3}{5}+\frac{2}{5r-13}+o(1)}.$
\end{prop}
These bounds follow from our results on $M(2r-1,r)$, which we determine precisely and $M(3r-4,r)$ which we determine up to lower order terms. This means that in terms of using random graphs as examples, these bounds are essentially best possible for $m=2r-1,3r-4$. We can obtain a constant factor improvement in the exponent for about $1/3$ of the range, but since we believe our current argument does not give the best possible answer, in terms of $M(m,r)$, for the whole range we leave this open for future research. 

The problem of determining $M(3r-4,r)$ is closely related to a well-studied problem of finding large independent sets in sparse triangle-free graphs. Perhaps the most famous result in this direction is due to Ajtai-Koml\'os-Szemer\'edi \cite {AKS} and Shearer \cite{shearer}, but for our problem earlier results of Staton \cite{staton} and Jones \cite{jones} turn out to be more relevant. These results are part of a very active research area of studying graphs having no cliques of size $k$ nor independent sets of size $r$ but which have potentially much fewer vertices than the corresponding Ramsey number $R(k,r)$. Our problem of lower bounding $M(m,r)$ falls under this framework since we can always assume that our graphs, in addition to having no independent sets of size $r$, are also $K_k$-free or the $2$-density is already large. 
We point the interested reader to classical papers \cite{staton,bollobas-tucker} and numerous papers citing them.

We now turn to the general case of $k \ge 4$. Let us begin with the initial instance, so when $r=3$. Unfortunately, here the results for $r=3$ do not immediately generalise as they did in the case of lower bounds. They do however provide a starting point, which serves as a basis for more general results. Here we determine $M(m,3)$ precisely for all $m$, which allows us to improve exponents in the previously best bounds of Linial and Rabinovich \cite{L-R}. They showed there are $n$-vertex graphs $G$ with $\alpha_m(G) \ge 3$ and $\alpha(G) \le n^{\frac{4+o(1)}{m-4/(m-2)}}$ if $m$ is even, and $\alpha(G) \le n^{\frac{4+o(1)}{m-3/(m-2)}}$ if $m$ is odd.

\begin{thm}\label{thm:main-ub-m-3}
For any $n$ there exists an $n$-vertex graph $G$ satisfying $\alpha_m(G) \ge 3$  with $\alpha(G) \le n^{\frac{4+o(1)}{m+2}}$ if $m$ is even, and $\alpha(G) \le n^{\frac{4+o(1)}{m+3-13/\sqrt{m}}}$ when $m \ge 5$ is odd $($here both terms $o(1)\to 0$ as $n\to \infty)$.
\end{thm}
We remark that in the even case any improvement of our exponent, in terms of $m$, provably leads to improvement over the best-known lower bounds on Ramsey numbers and in the odd case without improving the Ramsey numbers one can only improve the term $13/\sqrt{m}$. 

Once again our arguments in this particular regime show that the problem of determining $M(m,r)$ is related to yet another well-studied problem. Namely, the stability problem for Tur\'an's theorem first considered by Erd\H{o}s, Gy\H{o}ri and Simonovits in \cite{erdos-bip}. While one can use their results to obtain good bounds on $M(m,3)$ obtaining precise answers requires a different, more careful argument. 

In the fully general case we determine $M(m,r)$ up to lower order terms (where $r$ is considered fixed and $m$ large), giving us the following result. 

\begin{thm}\label{general-ub}
Let us assume $m$ is sufficiently larger than $r$ and set $k=\ceil{\frac m{r-1}}$. Then there exists $c_r>0$ such that for any $n$ there exists an $n$-vertex graph $G$ satisfying $\alpha_m(G) \ge r$ with $\alpha(G) \le n^{\frac{2+o(1)}{k+1-{c_r}/{\sqrt{k}}}}$.
\end{thm}

This improves previous upper bounds of Linial and Rabinovich from roughly $n^{\frac2{k-1}}$ to roughly $n^{\frac2{k+1}}$ and is once again essentially best possible assuming lower bounds on Ramsey numbers are tight.

While the above result requires $m$ to be large compared to $r$ some of our ideas apply for any choice of the parameters. To illustrate this we consider the case $m=20,r=5$ which was used by various researchers as a benchmark to compare their methods. Here Linial and Rabinovich show there are graphs $G$ having $\alpha_{20}(G) \ge 5$ and $\alpha(G) \le n^{18/39+o(1)}.$ Krivelevich \cite{michael-critical} improved this to $\alpha(G) \le n^{14/33+o(1)}.$ He obtains this as an application of his result on the minimum number of edges in colour-critical graphs. The best possible bound using this approach was later obtained by Kostochka and Jancey \cite{kost} who showed $\alpha(G) \le n^{18/43+o(1)}$. For comparison $39/18\approx 2.17, 33/14 \approx 2.36$ and $43/18\approx 2.39,$ while our methods allow us to improve this to $3$. That is, there exists a graph $G$ with $\alpha_{20}(G)\ge 5$ which has $\alpha(G) \le n^{1/3+o(1)}$ and once again this is best possible (up to the $o(1)$ term) without improving the lower bounds on Ramsey numbers $R(5,s)$. 

In addition to the above applications and connections, another reason which makes the study of $M(m,r)$ interesting is its relation to a random graph process. For a graph property $\mathcal{P}$ the random graph process with respect to $\mathcal{P}$ starts with an empty graph and iteratively adds a new uniformly random edge for as long as this does not violate $\mathcal{P}$. Random graph processes have been extensively studied for a variety of properties and have found numerous applications (see e.g. \cite{bohman-keevash, osthus-h-free, morris, nina-tash, matt-process,bohman-keevash2} and references therein). In our setting $M(m,r)$ controls the final density of the random process with respect to the \property{m}{r} $\alpha_m \ge r$. So $M(m,r)$ essentially controls the behaviour of this random process. 

\textbf{Organisation.} We will prove our lower bound results in \Cref{sec:lower bounds}. We start this section by proving \Cref{prop:k=3-lower-bound}. We continue by formalising in the form of \Cref{lem:reduction} the above-mentioned reduction of our general lower bound result, \Cref{thm:main-m-3}, to the case $r = 3$. We then focus on this case in \Cref{sec:2.1} to complete the proof of \Cref{thm:main-m-3}. In \Cref{sec:2.2} we prove our stronger bound in the $(m,r)=(7,3)$ case raised by Erd\H{o}s and Hajanl, namely \Cref{thm:main-7-3}. In \Cref{sec:2-density} we prove our upper bound results. We begin by proving our reduction to the $2$-density Tur\'an problem, namely \Cref{prop:m2-loc-ind}. We then switch the focus to proving our results concerning this $2$-density Tur\'an problem in \Cref{sec:turan-2density}. We begin by proving our results in the triangle-free regime in \Cref{sec:triangle-free-ub} providing us with a proof of \Cref{prop-alpha-k=3}. In \Cref{sec:independence-number-two} to solve the independence number two case which provides us with a proof of \Cref{thm:main-ub-m-3}. We complete the section with our asymptotic solution to the general problem which gives a proof of \Cref{general-ub}. In \Cref{sec:conc-remarks} we give some concluding remarks and open problems as well as the summary of our results. We also include three appendices. In \Cref{appendixC} we prove a slight modification of a result of \cite{jones} which we need in the proof of the second part of \Cref{prop-alpha-k=3}. In \Cref{appendixB} we completely solve a benchmark case $(m,r)=(20,5)$ of the Tur\'an $2$-density problem. Finally, in \Cref{appendixC} we prove some upper bounds for the Tur\'an $2$-density problem which establish the tightness of our results in this direction but are not necessary for our reduction to the local to global independence number problem. 

\textbf{Notation.} We will denote by $|G|$ the number of vertices in $G$ and by $e(G)$ the number of edges in $G$. For $v \in G$ we denote by $N(v)$ the neighbourhood of $v$ and by $d(v)=|N(v)|$ its degree in $G$. For $S\subseteq G$ we denote the induced subgraph of $G$ on this subset by $G[S]$. Whenever working with graphs satisfying $\alpha_m(G) \ge r$ all our asymptotics are with respect to $n=|G|$ and we treat $m$ and $r$ as constants unless otherwise specified. When working with directed graphs $N^{\pm}(v)$ denotes the in/out neighbourhood of $v$ and $d^{\pm}(v)$ in/out degree.

\section{Lifting the lower bounds from local to global independence number}\label{sec:lower bounds}
In this section, we prove our lower bounds on $\alpha(G)$ for a graph $G$ satisfying $\alpha_m(G) \ge r$. We begin with our lower bound result for $k=\ceil{\frac{m}{r-1}}=3$, namely \Cref{prop:k=3-lower-bound}.
\kisthree*

\begin{proof}
Let us first assume we can find a vertex disjoint collection of $a=t\ell-(r-1)$ cycles $C_{2\ell-1}$ and $t-a$ cycles $C_{2\ell+1}$. Their union makes a subgraph of $G$ of order $a(2\ell-1)+(t-a)(2\ell+1)=(2\ell+1)t-2a=2r-2+t=m$ and has no independent set of size larger than $a(\ell-1)+(t-a)\ell=t\ell-a=r-1$. This gives us a contradiction to $\alpha_m(G) \ge r$, which means such a union does not exist in $G$.

If we can find $a$ cycles $C_{2\ell-1}$ in $G$, then the remainder of the graph can not contain $t-a$ cycles $C_{2\ell+1}$, this means that by removing at most $m$ vertices from $G$ we can find a subgraph which is $C_{2\ell+1}$-free. If there are fewer than $a$ cycles $C_{2\ell-1}$ in $G$, then we find a subgraph, again with at least $n-m$ vertices which is $C_{2\ell-1}$-free. In either case, a classical result from \cite{erdos-cycles} (see also \cite{benny-cycles,li-cycles} for slight improvements) on cycle-complete Ramsey numbers tells us there is an independent set of size at least $\Omega((n-m)^{1-1/\ell})=\Omega(n^{1-1/\ell}),$ as desired.
\end{proof}

We now show how to generalise any improvement made in the $r=3$ case to half of the range for any $k$. It will be convenient to denote by $f(n,m,r)$ the smallest possible size of $\alpha(G)$ in an $n$ vertex graph with $\alpha_m(G) \ge r$. 
\begin{lem}\label{lem:reduction}
Let $k=\ceil{\frac m{r-1}}$ and $\ell=m-(k-1)(r-1)$. Provided $\ell \le \frac{r-1}{2}$ we have $f(n,m,r) \ge \min\{f(n-m,2k-1,3),f(n-m,k-1,2)\}$.
\end{lem}
\begin{proof}
Let $G$ be a graph on $n$ vertices with $\alpha_m(G) \ge r$. Let us first assume that we can find a vertex disjoint union consisting of $\ell$ subgraphs on $2k-1$ vertices, each having no independent set of size $3$, and $r-1-2\ell$ copies of $K_{k-1}$. This union is a subgraph on $\ell (2k-1)+(r-1-2\ell)(k-1)=(r-1)(k-1)+\ell=m$ vertices which has no independent set larger than $2\ell+r-1-2\ell=r-1$, contradicting $\alpha_m(G) \ge r$.

If we can find $\ell$ such subgraphs on $2k-1$ vertices this means that the remainder of the graph can not contain $r-1-2\ell$ copies of $K_{k-1}$. Removing our subgraphs and a maximal collection of $K_{k-1}$'s in the remainder we obtain a subgraph on at least $n-m$ vertices which is $K_{k-1}$-free, or in other words has $\alpha_{k-1} \ge 2$ implying $f(n,m,r) \ge f(n-m,k-1,2)$.
If there are fewer than $\ell$ such subgraphs we may remove a maximal collection and obtain a subgraph on at least $n-m$ vertices which contains no subgraphs on $2k-1$ vertices without an independent set of size $3$. In other words, our new subgraph has $\alpha_{2k-1} \ge 3$ so $f(n,m,r) \ge f(n-m,2k-1,3)$ as claimed.
\end{proof}

Note that $f(n,k-1,2)\ge \alpha$ is equivalent to $R(k-1,\alpha) \le n$ implying that the best known bound is 
\begin{equation}\label{eq:ramsey}
    f(n,k-1,2)\ge n^{\frac1{k-2}-o(1)}.
\end{equation} This represents a natural barrier for our results since it seems very likely that $f(n,2k-1,3)\ge f(n,k-1,2)$. On the other hand, the results of \cite{L-R} and  \cite{alon-sudakov} may be stated as $f(n,m,r) \ge f(n-m,k,2)$. Therefore, obtaining a lower bound for $f(n,2k-1,3)$, better than $f(n,k,2)$, immediately improves their bound whenever the above lemma applies, i.e.\ $\ell \le \frac{r-1}{2}$. Our result in the next section gives a bound which is halfway (in terms of exponents) between the above bounds coming from Ramsey numbers of $K_{k-1}$ and $K_k$ vs large independent set.

\subsection{Independent sets of size three everywhere.}\label{sec:2.1}
In this subsection, we will show how to find big independent sets in graphs satisfying $\alpha_{2k-1}(G) \ge 3$. This is inherently a Ramsey question in the following sense. How big a graph do we need to take in order to guarantee that we can find an independent set of size $\alpha$ or a subgraph $H$ with $2k-1$ vertices and no independent set of size $3$? The approach of \cite{L-R} and  \cite{alon-sudakov} is to always look for a single graph $H$, namely a vertex disjoint union of $K_k$ and $K_{k-1}$. This $H$ clearly has no independent set of size $3$ and the approach further reduces to finding a copy of $K_k$. The reason is that any bound strong enough to guarantee the existence of a $K_k$ will remain strong enough to force a $K_{k-1}$ once we remove the $k$ vertices of a single copy of $K_k$, and thus force a copy of $H$.
Our key new ingredient is to in addition look for a different $2k-1$-vertex graph which we call $H_{2k-1}$ and define to be a blow-up of $C_5$ with parts of sizes $1,k-2,1,1,k-2$ appearing in that order around the cycle, with cliques placed inside of parts (see \Cref{fig:0.1} for an illustration). Since the complement of this graph is an actual blow-up of $C_5$ it is triangle-free, implying that $H_{2k-1}$ has no independent set of size $3$ and is hence forbidden in any graph satisfying $\alpha_{2k-1}(G) \ge 3$.

\begin{figure}
\begin{minipage}[t]{0.265\textwidth}
\centering
\captionsetup{width=\textwidth}
\begin{tikzpicture}[rotate=-90]
\defPt{0}{0}{x0}
\defPt{1.25}{1.2}{x1}
\defPt{1.25}{0.4}{x2}
\defPt{1.25}{-0.4}{x3}
\defPt{1.25}{-1.2}{x4}
\defPt{2.5}{0.8}{x5}
\defPt{2.5}{-0.8}{x6}

\foreach \i in {1,...,4}
{
\draw[] (x0) -- (x\i);
}
\draw[] (x1) -- (x2);
\draw[] (x3) -- (x4);

\draw[] (x1) -- (x5);
\draw[] (x2) -- (x5);

\draw[] (x3) -- (x6);
\draw[] (x4) -- (x6);

\draw[] (x5) -- (x6);

\foreach \i in {0,...,6}
{
\draw[] (x\i) \smvx;
}

\end{tikzpicture}
\caption{$H_{7}$}
\label{fig:0.1}
\end{minipage}\hfill
\begin{minipage}[t]{0.28\textwidth}
\centering
\begin{tikzpicture}[rotate=-90]
\defPt{0}{0.75}{x0}
\defPt{1.5}{1}{x1}
\defPt{1.5}{0.5}{x2}
\defPt{1.5}{-0.5}{x3}
\defPt{1.5}{-1}{x4}
\defPt{2.75}{0.75}{x5}
\defPt{2.75}{-0.75}{x6}

\draw[] (x0) -- ($(x1)+(0,2)$);
\draw[] (x0) -- ($(x1)+(0,1)$);
\draw[] (x0) -- (2,0.75);
\draw[] (x0) -- ($(x4)-(0,0.5)$);
\draw[] (x0) -- ($(x4)-(0,-0.5)$);

\fitellipsis{$(x1)+(0,1.9)$}{$(x4)-(0,0.4)$}{0.7};

\draw[] (x5) -- ($(x1)-(0,0.1)$);
\draw[] (x5) -- ($(x2)-(0,-0.1)$);
\draw[] (x5) -- ($(x1)+(0,0.2)$);
\draw[] (x5) -- ($(x2)+(0,-0.2)$);

\draw[] (x6) -- ($(x3)-(0,0.1)$);
\draw[] (x6) -- ($(x4)-(0,-0.1)$);
\draw[] (x6) -- ($(x3)+(0,0.2)$);
\draw[] (x6) -- ($(x4)+(0,-0.2)$);

\fitellipsis{$(x1)+(0,0.2)$}{$(x2)-(0,0.2)$}{0.4};
\fitellipsis{$(x3)+(0,0.2)$}{$(x4)-(0,0.2)$}{0.4};

\node[] at ($(x1)+(0.8,1.5)$) {$N(v)$};
\node[] at ($(x0)+(0,0.35)$) {$v$};
\node[] at (1.5,0.025) {\tiny $\cdots$};
\node[] at (2.75,0.025) {\footnotesize $\cdots$};
\node[] at ($0.5*(x1)+0.5*(x2)$) {\footnotesize $K_{k-2}$};
\node[] at ($0.5*(x3)+0.5*(x4)$) {\footnotesize $K_{k-2}$};
\foreach \i in {0,5,6}
{
\draw[] (x\i) \smvx;
}
\end{tikzpicture}
\caption{$\mathcal{M}$}
\label{fig:0.2}
\end{minipage}\hfill
\begin{minipage}[t]{0.45\textwidth}
\centering
\begin{tikzpicture}[rotate=-90]
\defPt{0}{2.5}{x0}
\defPt{1.5}{1}{x1}
\defPt{1.5}{0.5}{x2}
\defPt{1.5}{-0.5}{x3}
\defPt{1.5}{-1}{x4}
\defPt{2.75}{0.75}{x5}
\defPt{2.75}{-0.75}{x6}

\draw[line width=1pt] ($(x0)-(0.25,0)$) -- ($(x1)+(0,2.5)$);
\draw[line width=1pt] ($(x0)+(1,0)$) -- ($(x1)+(0,2.5)$);
\draw[line width=1pt] ($(x0)+(1.75,0)$) -- ($(x1)+(0,2.5)$);
\draw[line width=1pt] ($(x0)+(3.1,0)$) -- ($(x1)+(0,2.5)$);

\fitellipsis{$(x0)-(0.15,0)$}{$0.5*(x1)+0.5*(x2)+(1.75,1.8)$}{0.5};

\draw [thick,decorate,decoration={brace,amplitude=5pt,raise=2pt},yshift=0pt]($(x0)+(2.25,0.55)$) -- ($(x0)+(3.3,0.3)$) node [black,midway,xshift=0.75cm,yshift=-0.05cm] {\small $ \ge x^2$};

\draw[line width=1pt] (x0) -- ($(x0)+(1.45,-0.2)$);

\draw[line width=1pt]  ($(x0)+(1.4,-0.2)$) --($(x0)+(2.8,0)$);

\draw[] ($(x0)+(1.45,-0.2)$) \smvx;
\draw[] ($(x0)+(2.8,0)$) \smvx;

\fitellipsisnfill{$(x1)+(0,3.5)$}{$(x4)-(0,0.4)$}{0.75};

\draw[] (x5) -- ($(x1)-(0,0.1)$);
\draw[] (x5) -- ($(x2)-(0,-0.1)$);
\draw[] (x5) -- ($(x1)+(0,0.2)$);
\draw[] (x5) -- ($(x2)+(0,-0.2)$);

\draw[] (x6) -- ($(x3)-(0,0.1)$);
\draw[] (x6) -- ($(x4)-(0,-0.1)$);
\draw[] (x6) -- ($(x3)+(0,0.2)$);
\draw[] (x6) -- ($(x4)+(0,-0.2)$);

\fitellipsis{$(x1)+(0,0.2)$}{$(x2)-(0,0.2)$}{0.4};
\fitellipsis{$(x3)+(0,0.2)$}{$(x4)-(0,0.2)$}{0.4};

\fitellipsis{$(x1)+(0,3.3)$}{$(x2)+(0,2.9)$}{0.35};

\node[] at ($(x4)+(-0.8,0)$) {$N(v)$};
\node[] at ($(x0)+(0.2,0.1)$) {$v$};
\node[] at (1.5,0.025) {\tiny $\cdots$};
\node[] at (2.75,0.025) {\footnotesize $\cdots$};
\node[] at ($0.5*(x1)+0.5*(x2)$) {\footnotesize $K_{k-2}$};
\node[] at ($0.5*(x3)+0.5*(x4)$) {\footnotesize $K_{k-2}$};
\node[] at ($0.5*(x1)+0.5*(x2)+(0,3.2)$) {\footnotesize $K_{k-3}$};
\foreach \i in {0,5,6}
{
\draw[] (x\i) \smvx;
}
\end{tikzpicture}
\caption{Extending $\mathcal{M}$}
\label{fig:0.3}
\end{minipage}
\end{figure}

We start by explaining the general idea behind our argument. As argued above our goal is to find a large independent set in an arbitrary $K_k$-free and $H_{2k-1}$-free $n$-vertex graph $G$. We will do so by finding a vertex $v$ and a large collection of vertex disjoint $K_{k-1}$'s with $k-2$ vertices inside $N(v)$. We know that the remaining vertex of any such $K_{k-1}$ lies outside $N(v)$ as our graph is $K_k$-free. Furthermore, we know that the set of these last vertices spans an independent set as otherwise any edge between such vertices together with their $K_{k-1}$'s and $v$ make a copy of $H_{2k-1}$ (see \Cref{fig:0.2} for an illustration). This gives us our desired large independent set. 

The more difficult part of the argument is to actually find such a collection of $K_{k-1}$'s. The following two easy lemmas will help us control how many $K_{k-1}$'s we can find with $k-2$ vertices inside a neighbourhood of $v$ and how many such $K_{k-1}$'s can intersect another one, respectively. Let us denote by $t_i(G)$ the number of copies of $K_i$ in $G$, (we omit $G$ when it is clear from context).

\begin{lem}\label{lem:turan-lb}
Let $G$ be a graph with $\alpha(G) \le \alpha$ and let $k \ge 2.$ Provided $t_{k-1}(G)>0,$ we have 
    $$\frac{t_k(G)}{t_{k-1}(G)} \ge \frac{|G|}{k\alpha^{k-1}}-1.$$
\end{lem}
\begin{proof}
We will prove the claim by induction on $k$. For the base case of $k=2$ (note that $t_1=|G|$) the claim follows from Tur\'an's theorem which gives $e(G) \ge \frac{n^2}{2 \alpha}-\frac{n}{2}$ (see \cite{alon-spencer}). Let us now assume $k \ge 3$ and that the claim holds for $k-1$. Given $S\subseteq V(G)$ let us denote by $e(S)$ the number of edges with both endpoints in the common neighbourhood of $S$ and by $d(S)$ the number of common neighbours of $S$. Then we have 
\begin{eqnarray*}
\binom{k}{2}t_k &=& \sum_{G[S]=K_{k-2}} e(S) \ge \sum_{G[S]=K_{k-2}} \left(\frac{(d(S))^2}{2\alpha}-\frac{d(S)}2\right)\\
&\ge& \frac{\left (\sum_{G[S]=K_{k-2}} d(S)\right)^2}{2t_{k-2}\alpha}-\sum_{G[S]=K_{k-2}}\frac{d(S)}2
=\frac{(k-1)^2t_{k-1}^2}{2t_{k-2}\alpha}-\frac{(k-1)t_{k-1}}2
\end{eqnarray*}
Where we used Turan's theorem within the common neighbourhood of $S$ in the first inequality and the Cauchy-Schwarz inequality for the second. Dividing by $(k-1)t_{k-1}/2$ and using the induction assumption
we obtain
$$\frac{kt_{k}}{t_{k-1}} \ge \frac{(k-1)t_{k-1}}{\alpha t_{k-2}}-1\ge \frac{|G|}{\alpha^{k-1}}- \frac{k-1}{\alpha}-1 \ge \frac{|G|}{\alpha^{k-1}}-k.$$

\vspace{-1cm}
\end{proof}

\begin{lem}\label{lem:turan-ub}
Let $G$ be a $K_k$-free graph with $\alpha(G) <  \alpha$. Then for any $i \le k$ we have
    $$t_{i}(G) \le \frac{1}{i!}\cdot \alpha^{\binom{k}{2}-\binom{k-i}{2}}.$$
\end{lem}

\begin{proof}
We prove the claim by induction on $i$. For the base case of $i=1$ the claim is equivalent to $t_1=|G| \le \alpha^{k-1}$ which holds by the classical bound on the Ramsey number $R(k,\alpha)$.  Let us now assume the claim holds for $i-1$. Given a subset of vertices $S$ we denote by $N(S)$ the set of common neighbours of $S$ and by $d(S)=|N(S)|$. If $G[S]=K_{i-1},$ then $N(S)$ is $K_{k-i+1}$-free in addition to having no independent set of size $\alpha$. So the same classical bound on Ramsey numbers as above implies $d(S) \le \alpha^{k-i}$. Taking a sum and using the inductive assumption we obtain: 
$$it_i=\sum_{G[S]=K_{i-1}}d(S) \le \alpha^{k-i}t_{i-1} \le  \frac{1}{(i-1)!}\cdot \alpha^{k-i+\binom{k}{2}-\binom{k-i+1}{2}}=\frac{1}{(i-1)!}\cdot \alpha^{\binom{k}{2}-\binom{k-i}{2}}.$$

\vspace{-1.2cm}
\end{proof}

We are now ready to prove our general result for $r=3$.

\begin{proof}[ of \Cref{thm:main-m-3}]
The main part of the proof will be to show that $f(n,2k-1,3)\ge \Omega(n^{1/(k-3/2)})$ but let us first verify this establishes the theorem in full generality. Indeed assuming this holds, given $k=\ceil{\frac{m}{r-1}}$ and $m \le (k-\frac12)(r-1)$ we have $m-(k-1)(r-1) \le \frac{r-1}{2}$ so that we may apply \Cref{lem:reduction} giving us, when combined with \eqref{eq:ramsey}, the desired result:
$$f(n,m,r) \ge \min\{f(n-m,2k-1,3),f(n-m,k-1,2)\}\ge \min\left\{\Omega(n^{\frac1{k-3/2}}),n^{\frac1{k-2}-o(1)}\right\}\ge \Omega\left(n^{\frac1{k-3/2}}\right).$$

Our remaining task now is to show that in an $n$-vertex graph $G$ which contains an independent set of size $3$ among any $2k-1$ vertices we can find an independent set of size $\alpha=\Omega(n^{1/(k-3/2)})$. With this choice of $\alpha$, we may assume that $n \ge C \alpha^{k-3/2}$ for an arbitrarily large constant $C$. Since the result for $k=3$ holds by \Cref{prop:k=3-lower-bound} we may assume $k \ge 4$. 

As discussed above if $G$ contains a $K_k$, then the remainder of the graph has no $K_{k-1}$, so by the classical Ramsey bound we get $\alpha(G) \ge (n-k)^{1/(k-2)} \ge \Omega(n^{1/(k-3/2)})$. So we may assume $G$ is $K_k$-free. 
Our goal is to find a vertex $v$ and a collection $\cal M$ of $\alpha$ vertex disjoint $K_{k-1}$'s each with $k-2$ vertices in $N(v)$. Then as we already explained above, any such $K_{k-1}$ has exactly one vertex outside $N(v),$ as our graph is $K_k$-free. These vertices outside of $N(v)$ form an independent set as otherwise any edge between such vertices together with their $K_{k-1}$'s and $v$ make a copy of $H_{2k-1}$, a contradiction.
In order to do so we will analyse common neighbourhoods of cliques of size $k-3$ inside $N(v)$. Any edge we find inside such a common neighbourhood gives rise to a copy of $K_{k-1}$ and if we find there a path of length $2$ starting with $v,$ then the last edge of this path gives rise to a copy of $K_{k-1}$ with exactly $k-2$ vertices in $N(v)$, which we are looking for. 

Let us first give a lower bound on $T:=(k-2)t_{k-2}$ (recall that $t_{k-2}$ denotes the number of $K_{k-2}$'s in $G$) which counts extensions of a $K_{k-3}$ into a $K_{k-2}$, i.e.\ the sum of sizes of common neighbourhoods of $K_{k-3}$'s in our graph. By a simple application of \Cref{lem:turan-lb} we get: 
\begin{equation}\label{eq:1}
    T=(k-2)t_{k-2} \ge t_{k-3}\cdot \left(\frac{n}{\alpha^{k-3}}-k+2\right)\ge t_{k-3} \cdot \frac{n}{2\alpha^{k-3}},
\end{equation} 
where in the second inequality we used $n\ge 2k\alpha^{k-3},$ which holds for our choice of $\alpha$. Using \Cref{lem:turan-lb} repeatedly in a similar way we get the following lower bound on $T$ which will be useful later in the argument.
\begin{equation}\label{eq:2}
    T=(k-2)t_{k-2} =(k-2)\cdot \frac{t_{k-2}}{t_{k-3}}  \cdots  \frac{t_{2}}{t_{1}} \cdot n \ge  (k-2)\cdot \frac{n}{2(k-2)\alpha^{k-3}} \cdots  \frac{n}{4 \alpha} \cdot n=\frac{n^{k-2}}{2^{k-3}(k-3)!\alpha^{\binom{k-2}{2}}}
\end{equation}
In order to carry out the above proof strategy, for every clique of size $k-3$ we are going to restrict our attention only to a large part of its common neighbourhood where independent sets expand (meaning they have many vertices adjacent to some vertex of the set). This will achieve our goal since $G$ being $K_k$-free means that the common neighbourhood of a $K_{k-3}$ is triangle-free, so the part of this common neighbourhood inside $N(v)$ is an independent set. Hence, expansion of this particular independent set precisely means there are many endpoints of a path of length $2$ starting with $v$, giving many choices for a $K_{k-1}$ with $k-2$ vertices in $N(v)$. 

To obtain such an expansion we are going to first restrict attention to $K_{k-3}$'s which have the common neighbourhood of order at least half the average which equals $T/t_{k-3}$. We will call such a $K_{k-3}$ \textit{typical} and we know that altogether there are at least $T/2$ ways of extending typical $K_{k-3}$'s into a $K_{k-2}$. By \eqref{eq:1}, any typical $K_{k-3}$ has a common neighbourhood of size at least $d:= n/(4\alpha^{k-3})$.

We now proceed to obtain independent set expansion inside the common neighbourhood of each typical $K_{k-3}$. Let $S$ be such a $K_{k-3}$. Let $X$ be a maximal independent set inside of the common neighbourhood $N$ of $S$ which has fewer than $\frac{d-2\alpha}{2\alpha} \cdot |X|$ neighbours inside $N$, in case no such sets exists we set $X=\emptyset$. We now remove $X$ and all its neighbours from $N$. Since $X$ is an independent set we know $|X|\le \alpha$ so we have removed at most $\alpha+\frac{d-2\alpha}{2\alpha}\cdot \alpha = \frac d2$ vertices from $N$. Any remaining vertex in $N$ is called an \textit{expanding neighbour} of $S$. Since $S$ was arbitrary, every typical $K_{k-3}$ has at least $d/2$ expanding neighbours and inside its expanding neighbourhood independent sets expand by a factor of $\frac{d-2\alpha}{2\alpha} \ge \frac{d}{4\alpha}=\frac{n}{16\alpha^{k-2}}=:d'$. In particular, every vertex in $N$ (being an independent set of size one) has degree at least $d'$ in this set. Furthermore, there are still at least $T/4$ ways to extend a typical $K_{k-3}$ into a $K_{k-2}$ using an expanding neighbour. 

We now pick our $v$ to be a vertex which is an expanding neighbour of $g_v \ge T/(4n)$ typical $K_{k-3}$'s (such $v$ exists by double counting and the above bound). Let $\mathcal{S}$ denote the collection consisting of all such $K_{k-3}$'s.
Let us first observe some properties of an $S \in \mathcal{S}$. We denote by $N_S$ its expanding neighbourhood and by $D_S:=N(v) \cap N_S$ the set of its expanding neighbours inside $N(v)$. By definition, $v \in N_S$ for all $S \in {\mathcal S}$. 
Also, as was explained above, we know that $v$ has at least $d'$ neighbours within $N_S$, i.e.  $|D_S| \geq d'$.
These neighbours span an independent set (since they belong to the common neighbourhood of $k-2$ vertices in $v \cup S$) of size at least $d'$, so they expand inside $N_S$. This gives us $d'^2$ different vertices which together with $S$ and one of the vertices in $D_S$ make a $K_{k-1}$ with exactly $k-2$ vertices inside $N(v)$. To find many such disjoint $K_{k-1}$'s we will use the fact that there are in total $\sum_{S \in \mathcal{S}} |D_S| \ge g_v \cdot d'$ ways to extend $K_{k-3}$'s in $\mathcal{S}$ into a $K_{k-2}$ using an expanding neighbour belonging to $N(v)$.

Let us now consider a maximal collection $\mathcal{M}$ of vertex disjoint $K_{k-1}$'s each with exactly $k-2$ vertices inside $N(v)$. Let us assume towards a contradiction that $|\mathcal{M}| < \alpha$. We will show below that if $|\mathcal{M}| < \alpha,$ we can still find some $S \in \mathcal{S}$ and a set $D_S' \subseteq D_S$ of at least $d'/2$ of its expanding neighbours in $N(v)$ such that both $S$ and $D'_S$ are vertex disjoint from all cliques in $\cal M$. For now, suppose we found such $S$ and $D_S'$. 
Since $D_S'\subseteq D_S$ is an independent set (as we explained above), it expands within the neighbourhood of $S$ meaning that there are at least $d'^2/2=\frac{n^2}{2^9\alpha^{2k-4}} \ge \alpha$ vertices which together with $S$ and some vertex in $D_S'$  make a $K_{k-1}$ with $k-2$ vertices in $N(v)$. Moreover, note that all these $d'^2/2$ vertices lie outside of $N(v)$ or we get a $K_k$ in $G$. Since we removed fewer than $\alpha$ vertices outside of $N(v)$ (recall that each $K_{k-1} \in \mathcal{M}$ has exactly one vertex outside $N(v)$) one of these vertices is disjoint from all cliques in $\mathcal{M}$ and gives rise to the desired copy of $K_{k-1}$, which is vertex disjoint from any clique in $\mathcal{M}$ (since both $S$ and $D_S'$ are chosen disjoint from any clique in $\mathcal{M}$) and hence contradicts its maximality.

Therefore, it remains to be shown that there is an $S \in \mathcal{S}$ and $d'/2$ of its expanding neighbours inside $N(v)$, all disjoint from any clique in $\mathcal{M}$. Note that cliques in $\mathcal{M}$ cover at most $\alpha (k-2)$ vertices inside $N(v)$. On the other hand note that any vertex $u\in N(v)$ can belong to at most $\alpha^{\binom{k-2}{2}}/(k-3)!$ copies of $K_{k-2}$ inside $N(v)$. This follows from \Cref{lem:turan-ub} since any such copy of $K_{k-2}$ amounts to a copy of $K_{k-3}$ in the common neighbourhood of $v$ and $u$ which spans a $K_{k-2}$-free graph with no independent set of size $\alpha$ (or we are done). On the other hand, any $K_{k-2}$ can be an extension of at most $k-2$ different copies of $K_{k-3}$ in $\mathcal{S}$ so there are at least 
$g_v d'-\alpha^{\binom{k-2}{2}+1}(k-2)^2/(k-3)!$ extensions of a $K_{k-3}$ from $\mathcal{S}$ to a $K_{k-2}$ inside $N(v)$ both disjoint from any clique in $\mathcal{M}$. Note that 
\vspace{-0.4 cm}
\begin{align*}
    & \frac{g_v d'}{\alpha^{\binom{k-2}{2}+1}(k-2)^2/(k-3)!} \ge \frac{n^{k-3}/(2^{k-1}\alpha^{\binom{k-2}{2}}) \cdot  n/(16 \alpha^{k-2})}{\alpha^{\binom{k-2}{2}+1}(k-2)^2}=\frac{n^{k-2}}{2^{k+3}(k-2)^2\alpha^{(k-2)^2+1}} \ge \frac{1}{2} 
\end{align*}
where in the first inequality we used $g_v \ge T/(4n)$ and \eqref{eq:2} to bound $T$, while in the last inequality we used that $n \ge C\alpha^{k-3/2} \ge C\alpha^{k-2+1/(k-2)}.$ This means that there are at least $g_vd'/2$ such extensions and since $g_v=|\mathcal{S}|$ there must be a $K_{k-3}$ with $d'/2$ extensions, as desired.
\end{proof}

\subsection{The Erdos-Hajnal (7,3) case. }\label{sec:2.2}

For $k=4$ the result from the previous section implies that graphs with $\alpha_7 \ge 3$ have $\alpha \ge \Omega(n^{2/5})$ which already suffices to confirm the conjecture of Erd\H{o}s and Hajnal \cite{erdos-problems}. In this section, we show how to further improve this bound to $\alpha \ge n^{5/12-o(1)}$, i.e.\ we prove \Cref{thm:main-7-3}.

The general idea will be similar as in the previous subsection. Here since $k=4$ we may assume our $G$ satisfying $\alpha_7(G)\ge 3$ is $K_4$ and $H_7$-free. In fact, in this case, $\alpha_7(G) \ge 3$ is essentially (up to removal of a few vertices) equivalent to $G$ being $K_4$ and $H_7$-free. Since we do not need the non-obvious direction here, we prove it as \Cref{lem:equivalence} in the following section where it will be useful. 

Unlike in the previous section, since the desired $\alpha$ is bigger we will not be able to find a large enough set of vertex disjoint triangles (not containing $v$) with an edge in $N(v)$. However, these triangles will still play a major role in the argument. We will call them \textit{$v$-triangles} and the vertex of a $v$-triangle not in $N(v)$ is going to be called a $v$-extending vertex (in other words any non-neighbour of $v$ which belongs to a $K_4$ minus an edge together with $v$). While we can not find a large enough collection of disjoint $v$-triangles the fact our graph is $H_7$-free imposes many restrictions on the subgraph induced by $v$-extending vertices, which we call $N_{\triangle}(v)$. The following lemma establishes the properties of $N_{\triangle}(v)$ that we will use in our argument. 

\begin{lem}\label{lem:properties}
Let $G$ be a $K_4$ and $H_7$-free graph and $v \in G$. Then
\begin{enumerate}[label=\alph*)]
    \item $N_{\triangle}(v) \cap N(v) = \emptyset$.
    \item If $u,w \in N_{\triangle}(v)$ belong to vertex disjoint $v$-triangles, then $u \nsim w.$
    \item $N_{\triangle}(v)$ is triangle-free.
    \item Let $C$ be a connected subgraph of $N_{\triangle}(v)$ consisting only of vertices belonging to at least $5$ different $v$-triangles. If $|C| \ge 2,$ then there exists $u\in N(v)$ such that for any $w\in C,$ all the edges within $N(v) \cap N(w)$ make a star centred at $u$ in $G$.
\end{enumerate}
\end{lem}

\begin{proof}
Part a) is immediate since $G$ is $K_4$-free and part b) since it is $H_7$-free.

For part c) assume to the contrary that there is a triangle $x,y,z$ in $N_{\triangle}(v)$ and that $f,g,h$ are edges in $N(v)$ completing a $v$-triangle with $x,y,z$ respectively. By part b) any two of $f,g,h$ need to intersect. This is only possible if they make a triangle, in which case together with $v$ they make a $K_4$, or if they make a star, in which case the centre of the star together with $x,y,z$ makes a $K_4$, either way we obtain a contradiction.

For part d) let $u,w \in N_{\triangle}(v)$ be adjacent and belong to at least $5$ different $v$-triangles. We claim that then edges within $(N(w) \cap N(v)) \cup (N(u) \cap N(v))$ make a star in $G$. Indeed if $N(u)\cap N(v)$ would contain two disjoint edges, then by part b) any edge in $N(w) \cap N(v)$, of which there are at least $5$ by assumption, must intersect them both. Since there can be at most $4$ edges which intersect both of the two disjoint edges, we conclude there can be no disjoint edges in $N(u)\cap N(v)$. This means that the edges span either a star or a triangle. Since there are at least $5$ edges, the former must occur. We can repeat for $w$ in place of $u$ and observe that the only way for each pair of edges, one per star, to intersect is that they share the centre, as claimed. Propagating along any path in $C$ we deduce that the same holds for any pair of vertices in $C$.
\end{proof}

The following corollary, based mostly on part d) of the above lemma, allows us to partition $N_\triangle(v)$ into three parts which we will deal with separately in our argument.

\begin{cor}\label{cor:partition}
Let $G$ be a $K_4$ and $H_7$-free graph and $v \in G$. Then there exists a partition of $N_\triangle(v)$ into three sets $L,I$ and $C$ with the following properties:
\begin{enumerate}[label=\alph*)]
    \item $L$ consists only of vertices belonging to at most $4$ different $v$-triangles. 
    \item $I$ is an independent set.
    \item $C$ can be further partitioned into $C_1,\ldots, C_m$ such that there are no edges between different $C_i$'s and for every $C_i$ there is a distinct $v_i \in N(v)$ such that any $v$-triangle containing a vertex from $C_i$ must contain $v_i$ as well.
    \end{enumerate}
\end{cor}

\begin{proof}
We chose $L$ to consist of all $v$-extending vertices belonging to at most $4$ $v$-triangles. We chose $I$ to consist of isolated vertices in $G[N_{\triangle}(v) \setminus L]$ and $C=N_{\triangle}(v) \setminus (L \cup I)$. Note that $C$ is a union of connected components of $G[N_{\triangle}(v) \setminus L]$, which we denote by $C_1',C_2',\ldots$ each of order at least $2$ and consisting entirely of vertices belonging to at least $5$ different $v$-triangles. In particular, \Cref{lem:properties} part d) implies that there exists a vertex $v_i'\in N(v)$ such that any $v$-triangle containing a vertex from $C_i'$ must also contain $v_i'$. Finally, we merge any $C_i'$s which have the same vertex for their $v_i'$ to obtain the desired partition $C_1,\ldots, C_m$ of $C.$ 
\end{proof}

\vspace{-0.4cm}
We are now ready to prove \Cref{thm:main-7-3}.
\vspace{-0.2cm}
\begin{proof}[ of \Cref{thm:main-7-3}]
Let $G$ be an $n$-vertex graph with $\alpha_7(G) \ge 3$. Our task is to show it has an independent set of size $\alpha=n^{5/12-o(1)}.$ If $G$ contains a $K_4$ the remainder of the graph must be triangle-free so has an independent set of size $\Omega(\sqrt{n})>\alpha$. Hence, we may assume $G$ is $K_4$-free as well as $H_7$-free.

We begin by ensuring the minimum degree is high, so that we can ensure good independent set expansion inside neighbourhoods. We repeatedly remove any vertex with degree at most $2n/\alpha.$ Observe that if we remove more than half of the vertices, then the removed vertices induce a subgraph with at least $n/2$ vertices and at most $n\cdot {2n}/{\alpha}$ edges. So, Tur\'an's theorem implies there is an independent set of size at least  $\frac{n/2}{8n/\alpha+1}\ge \Omega(\alpha)$ and we are done. Let us hence assume that $G$ has minimum degree at least $2n/\alpha$ (we technically need to pass to a subgraph on at least $n/2$ vertices but this only impacts the constants).

Let us fix a vertex $v$. Let $X$ be a maximal independent set inside $N(v)$ which has fewer than $n/\alpha^2 \cdot |X|$ neighbours inside $N(v)$, and we set $X=\emptyset$ if such a set does not exist. Since $X$ is independent we may assume it has size at most $\alpha/2$ so $|X|+|N(X) \cap N(v)| \le \alpha/2+n/\alpha^2\cdot \alpha/2 \le n/\alpha$. We direct\footnote{The assignment of directions is simply a convenient way to encode the information about in what part of $N(v)$ we know independent sets expand. An out-neighbour corresponds to an expanding neighbour in the previous argument.} all edges from $v$ towards $N(v)\setminus (X \cup N(X))$. Note that $d^+(v) \ge d(v)/2\ge n/\alpha$ and inside $N^+(v)$ independent sets expand by at least a factor of $x:=n/\alpha^2\ge \Omega(n^{2/12})$. We repeat for every vertex $v$, and note that some edges of $G$ might be assigned both directions, while some other edges none.

Our first goal is to show there are in total many $v$-triangles for which their edge in $N(v)$ belongs to $N^{-}(v)$ and the remaining two edges of the triangle are directed away from this edge. We will call such a $v$-triangle a \textit{directed $v$-triangle}. In other words, we are counting the number of $K_4$'s minus an edge with the $4$ edges incident to the missing edge all being directed towards vertices of the missing edge. We denote this count by $T_4$.  

\begin{claim*}
Unless $\alpha(G) \ge \Omega(n^{5/12})$ we have $T_4 \ge \Omega(n^2).$
\end{claim*}

\begin{proof}
We will show that for any vertex $v$ with in-degree at least half the average we get at least $\Omega(d^-(v)^2/n^{2/12})$ directed $v$-triangles. Let us for now assume this holds. Note that $v$ with lower in-degree contribute at most half to the value of $\sum_{v \in G} d^{-}(v)=\sum_{v \in G} d^{+}(v) \ge n^2/\alpha \ge \Omega(n^{19/12})$ (recall that $d^{+}(v) \ge n/\alpha$). An application of Cauchy-Schwarz implies there are at least $\Omega((\sum_{v \in G} d^{-}(v))^2/n \cdot n^{-2/12})\ge\Omega(n^2)$ directed triangles.

Let us now fix a vertex $v$ with in-degree at least half the average, this in particular means $d^{-}(v) \ge n/(2\alpha)\ge n^{7/12}$. We know that any vertex $u\in N^{-}(v)$ has $v$ as an out-neighbour which means that $v$, as a single vertex independent set, expands inside $N^{+}(u)$. So $v$ has at least $x$ neighbours inside $N^+(u),$ i.e.\ $u$ has at least $x$ out-neighbours inside $N(v)$. Since $u$ was an arbitrary vertex in $N^{-}(v)$ this means there are at least $d^{-}(v)x\ge \Omega(n^{9/12})$ edges inside $N(v)$ which are directed away from a vertex in $N^{-}(v)$ (note that if an edge of $G$ has both directions and is inside of $N^{-}(v),$ then it is counted twice and considered as $2$ directed edges). Let us call the set of such directed edges $M$, so in particular $|M| \ge d^{-}(v)x \ge \Omega(n^{9/12})$. Note further that given such a directed edge $uw \in M$, since $u$ has $w$ as a single vertex independent set in its out-neighbourhood $w$ expands there. This means that our edge lies in at least $x$ $v$-triangles with both edges incident to $u$ directed away from $u$. We call the third vertex of any such triangle $uw$-extending, note that it belongs to $N_{\triangle}(v)$. We are now going to assign types to edges in $M$ according to where, inside $N_\triangle(v)$, we find the majority of their extending vertices.
 
Let us fix a partition of $N_\triangle(v)$ into $C,L$ and $I,$ provided by \Cref{cor:partition}. Given a directed edge $uw \in M$ we say it is of type $L$ or $C$ if it has at least $x/3$ extending neighbours in $L$ or $C$, respectively. We say it is of type $I$ if it was not yet assigned a type. In particular, an edge of type $I$ also has at least $x/3$ extending vertices in $I$ (since we have shown above it has at least $x$ in total), but we also know it has at most $2x/3$ extending vertices in other parts of $N_\triangle(v)$.

\textbf{First case:} at least a third of the edges in $M$ are of type $C$.\\ 
Any vertex in $N(v)$ is adjacent to fewer than $\alpha$ other vertices inside $N(v)$ (since these vertices make an independent set as $G$ is $K_4$-free). This means there is a matching $M'$ of at least $|M|/(6\alpha)$ edges of type $C$, as otherwise, vertices making a maximal matching are incident to fewer than $|M|/3$ edges so it can be extended. Let us denote by $N_e$ the set of extending vertices of an edge $e\in M'$ inside $C$. By assumption, $|N_e|\ge x/3$ and note that each $N_e$ spans an independent set (being in a common neighbourhood of $e$). On the other hand, we also claim $N_e$'s are disjoint. To see this suppose $x \in N_e \cap N_{e'}$ for two distinct edges $e,e' \in M'$. By definition of $N_e$, there is some $i$ such that $x \in C_i$ and by \Cref{cor:partition} part c) we know that $v_i \in e$ and $v_i \in e'$, a contradiction. Note also that there can be no edges between distinct $N_e,N_{e'}$ or we find an $H_7$. This means that $\bigcup_{e \in M'} N_e$ is an independent set of size at least $x/3 \cdot |M|/(6\alpha)=\Omega(n^{6/12}).$

\textbf{Second case:} at least a third of the edges in $M$ are of type $L$. \\ 
Since any vertex in $L$ belongs to at most $4$ distinct $v$-triangles it can in particular be an extending vertex of at most $4$ edges from $M$. Since every edge of type $L$ has at least $x/3$ extending vertices in $L$ this means $|L| \ge x|M|/12 \ge n^{11/12}/12$. Now, \Cref{lem:properties} part c) implies $L\subseteq N_{\triangle}(v)$ is triangle-free so there is an independent set of size at least $\sqrt{|L|} \ge \Omega (n^{5.5/12}).$
    
\textbf{Third case:} at least a third of the edges in $M$ are of type $I$.\\ 
    Let $u \in N^{-}(v)$. Let us denote by $T_u$ the set of out-neighbours of $u$ which together with $u$ make an edge of type $I$. We know by the case assumption that $\sum_{u \in N^{-}(v)} |T_u| \ge |M|/3$. Note that $T_u\subseteq N(v) \cap N^{+}(u)$ so must span an independent set (or we find a $K_4$ in $G$) and hence expands inside $N^{+}(u)$. This means there are at least $x|T_u|$ distinct vertices extending an out-edge of $u$ of type $I$. Note however that we do not know that all of them must be in $I$. But since any edge of type $I$ has at most $2x/3$ extending neighbours outside of $I$ this means that there are least $x|T_u|/3$ vertices in $I$ extending an out-edge of $u$. This in particular means that $u$ sends at least $x|T_u|/3$ edges directed towards $I$. By taking the sum over all $u$ we obtain that the number of edges directed from $N^{-}(v)$ to $I$ is at least $xM/9 \ge d^{-}(v)x^2/9 \ge \Omega(n^{11/12})$.
    
    Since $I$ spans an independent set we may assume $|I| < \alpha$. Let $S_u=N^{-}(v) \cap N^{-}(u)$ for any $u\in I$. Let $s_u=|S_u|$ so we know that $\sum_{u \in I} s_u \ge d^{-}(v)x^2/9\ge \Omega(n^{11/12})$. Let $I'$ be the subset of $I$ consisting of vertices $u$ with $s_u\ge 2\alpha$. Since vertices of $I \setminus I'$ contribute at most $2|I|\alpha<n^{10/12}$ to the above sum, we still have $\sum_{u \in I'} s_u \ge d^{-}(v)x^2/18.$ By Tur\'an's theorem for any $u \in I'$ there needs to be at least $s_u^2/(4\alpha)$ (using that $s_u \ge 2\alpha$) edges inside $S_u$, or we find an independent set of size $\alpha$. Each such edge gives rise to a directed $v$-triangle. Hence, using Cauchy-Schwarz, there are at least $\sum_{u \in I'} {s_u^2}/{(4\alpha)} \ge \Omega(d^-(v)^2x^4/\alpha^2)\ge \Omega(d^-(v)^2/n^{2/12})$ directed $v$-triangles, as desired.
\end{proof}

Let us give some intuition on how we are going to use the fact that $T_4$ is big. Let us denote by $\overrightarrow{d_e}$ the number of common out-neighbours of vertices making an edge $e$. Recall that $T_4$ can be interpreted as the number of $K_4$'s minus an edge with all its edges oriented towards the missing edge. We call the remaining edge (for which we are not insisting on the direction) the \textit{spine}. Note that the number of our $K_4$'s minus an edge having some fixed edge $e$ as the spine is precisely $\binom{\overrightarrow{d_e}}{2}$. This means that $T_4=\sum_{e \in E(G)} \binom{\overrightarrow{d_e}}{2}.$ Our bound on $T_4$ obtained above tells us that in a certain average sense the $\overrightarrow{d_e}$'s should be big. 

Observe now that if one finds a star centred at $v$ consisting of $s$ edges each with $\overrightarrow{d_e} \ge t,$ then this means that the $s$ leaves each have $t$ out-neighbours inside $N(v)$ (note that we are disregarding the information that they are in fact inside $N^{+}(v)$). This will allow us to play a similar game as we did in the previous claim, indeed there we tackled the same problem with $s=d^{-}(v)$ and $t=x$ (which we obtained through expansion) with an important difference, namely that the star was in-directed. The bound on $T_4$ tells us that there is such a star with larger (in certain sense) parameters $s$ and $t$ and the next claim shows how this gives rise to many of our $K_4$'s minus an edge, which we find in a different place (in particular they do not use the centre of the star since we do not know the direction of centre's edges). 

For any $v \in G$ let $S_v$ be the star consisting of a centre $v$ and all its edges in $G$. Let us also denote by $s_v$ the total number of our $K_4$'s minus an edge with an edge of $S_v$ as their spine, i.e.\ $s_v=\sum_{e \in S_v} \binom{\overrightarrow{d_e}}{2}$. Summing over $v$ we also have  $2T_4=\sum_{v}\sum_{e \in S_v} \binom{\overrightarrow{d_e}}{2}=\sum_{v} s_v$.

\begin{claim*}
Unless $\alpha(G) \ge n^{5/12-o(1)}$ and provided $s_v \ge \frac{T_4}{n}$ there exist $\frac{s_v^2}{n^{7/12+o(1)}}$ of our $K_4$'s minus an edge with their spine inside $N(v)$.
\end{claim*}

\begin{proof}
Let us first ``regularise'' $\overrightarrow{d_e}$'s for $e \in S_v$. Let us partition these edges into at most $\log n$ sets with all edges belonging to a single set having $\overrightarrow{d_e}  \in [t,2t]$ for some $t$. Since $s_v=\sum_{e \in S_v} \binom{\overrightarrow{d_e}}{2}$ and the edges in $S_v$ are split into at most $\log n$ sets, we conclude that one set contributes at least $\frac{s_v}{\log n}$ to this sum. I.e.\ edges in this set make a substar $S$ of $S_v$ consisting of $s$ edges, each having $t \le \overrightarrow{d_e} \le 2t$ for some $s,t$ satisfying $2st^2 \ge s\binom{2t}{2} \ge \frac{s_v}{\log n} \ge \frac{T_4}{n\log n} \ge n^{1-o(1)}$.\footnote{We decided to pay the $\log n$ factor here for simplicity, it is possible to do the same argument more carefully and avoid it.}
We may assume that $t \le n^{5/12}$ since $\overrightarrow{d_e}$ counts certain common neighbours of a fixed edge which must span an independent set (or there is a $K_4$).

Similarly, as in the previous claim, we define $M$ to be the set of directed edges inside $N(v)$ with source vertex being a leaf of $S$. Now the fact that any edge $e \in S$ has $\overrightarrow{d_e} \ge t$ means that any leaf $u$ of $S$ is a source of at least $t$ edges in $M$. Let us remove all but exactly $t$ such edges from $M$, so, in particular, $|M|=st$ (as before while some edges might be oriented both ways we treat this as two distinct directed edges). Let us denote by $T_u$ the set of $t$ out-neighbours of $u$ which together with $u$ make an edge in $M$. We know $T_u$ is an independent set (it consists of common neighbours of the edge $vu$). In particular, both $T_u$ and any of its subsets expand inside $N^{+}(u)$. Let us consider an auxiliary bipartite graph with the left part being $T_u$ and the right part being $N_\triangle(v)\cap N^{+}(u)$. We put an edge between two vertices if together with $u$ they make a triangle in $G$. The expansion property translates to the fact that any subset of size $t'$ of the left part has at least $t'x$ distinct neighbours on the right. A standard application of Hall's theorem (to the graph obtained by taking $x$ copies of every vertex on the left) tells us we can find $t$ disjoint stars each of size $x$ in this graph. Translating back to our graph, for each of the $t$ edges incident to $u$ in $M$ we have found a set of $x$ out-neighbours of $u$ which extend it into a triangle. Moreover, these sets are disjoint for distinct edges. We call these $x$ vertices extending for the corresponding edge\footnote{ Note that this is a subset of what we considered to be extending vertices in the previous claim. Here it is important for us to fix the number of extending neighbours for every edge for certain regularity considerations.}. 

Once again let us take a $C,L,I$ partition provided by \Cref{cor:partition} and assign types $C,L$ and $I$ to edges in $M$, which have at least $x/3$ extending neighbours in $C,L$ and $I$, respectively. This is similar to the previous argument except that we are using our new, slightly modified definition of extending vertices. We again split into three cases according to which type is in the majority. 

\textbf{First case:} at least $st/3$ edges in $M$ are of type $C.$\\ 
As in 1.\ case of the previous claim we can find a matching $\mathcal{M}$ of at least $st/(6\alpha)$ edges of type $C$ (since vertices of $\mathcal{M}$ are incident to at most $2|\mathcal{M}|\alpha$ edges inside $N(v)$). 
Let us denote by $N_e$ the set of extending neighbours of an edge $e \in \mathcal{M}$ inside $C$, so $|N_e| \ge x/3$. Again, as before, each $N_e$ is independent (neighbours of the same edge), and they are disjoint for different $e$ (otherwise if a vertex belonging to two $N_e$'s belongs to $C_i,$ we get a contradiction to the uniqueness of $v_i$) and there are no edges between distinct $N_e$'s (or we find an $H_7$). So their union makes an independent set of size at least $stx/(18\alpha)$. If $t \le n^{4/12},$ then $st^2 \ge n^{1-o(1)}$ implies $st \ge n^{8/12-o(1)}$ and our independent set is of size at least $n^{5/12-o(1)}$, as desired. So let us assume $t \ge n^{4/12}.$ 

Note that by \Cref{cor:partition}, given an edge $e\in M$ and its extending neighbour which belongs to some $C_i$ we know that $v_i \in e$ and $v_i$ can be either source or sink of $e$. In the former case, we say the neighbour is source-extending and in the latter sink-extending. We say $e$ is of type $C$-source if it has at least $x/6$ source-extending neighbours in $C$ and of type $C$-sink if it has at least $x/6$ sink-extending neighbours in $C$.

If there are at least $st/6$ edges in $M$ of type $C$-sink that means there is a leaf $u$ of $S$ which is a startpoint of at least $t/6$ such edges. If $uw$ is one of these $t/6$ edges it has a set of at least $x/6$ sink-extending neighbours which span an independent set (being common neighbours of an edge) and all belong to the same $C_i$ (namely the one for which $v_i=w$) and these $C_i$'s are distinct between edges. This means that these sets are disjoint between ones corresponding to distinct edges and span an independent set (there are no edges between distinct $C_i$'s) so we found an independent set of size at least $tx/36 \ge \Omega(n^{6/12})$.

If there are at least $st/6$ edges in $M$ of type $C$-source we need to be able to find $s/12$ leaves of $S$ incident to at least $t/12$ such edges each (since we know that every leaf of $S$ is incident to exactly $t$ edges in $M$ so otherwise, there would be fewer than $s/12 \cdot t + s \cdot t/12=st/6$ such edges in total). For any such leaf $u$ this means we find $\frac{t}{12} \cdot \frac{x}{6} \ge \frac{tx}{72}$ source-extending neighbours of its edges (using our preprocessing fact that extending neighbours of distinct edges incident to $u$ are disjoint). By \Cref{lem:properties} c) we know there is an independent set of size $\sqrt{tx}/9$ among these neighbours. Since these are source-extending neighbours we know they all belong to $C_i$ for which $v_i=u$. In particular, for distinct $u$ they belong to distinct $C_i$'s, meaning we obtain an independent set of size $\Omega(s\sqrt{tx})=\Omega(st^2\sqrt{x}/t^{3/2})\ge n^{13/12-o(1)}/t^{3/2} \ge n^{5.5/12-o(1)},$ using $t \le n^{5/12}$.

\textbf{Second case:} at least $st/3$ edges in $M$ are of type $L$.\\ 
Let us first assume $s \ge t$. We again find a matching of size $st/(6\alpha)$ of edges of type $L$ in $M$. Each edge $e$ in the matching gives rise to a set $N_e$ of $x/3$ extending neighbours in $L$. $N_e$ spans an independent set (being inside the common neighbourhood of an edge) and there can be no edges between distinct $N_e$'s (or we find an $H_7$). This means that the union of $N_e$'s spans an independent set. Since any extending neighbour in this union can be extending for at most $4$ edges (so belongs to at most $4$ different $N_e$'s) this gives $\alpha(G) \ge stx/(72\alpha) \ge (st^2)^{2/3}x/(72\alpha) \ge n^{5/12-o(1)},$ using $s \ge t$ and $st^2 \ge n^{1-o(1)}$.

Let us now assume $t \ge s$. We can find at least $s/6$ leaves of $S$ each being a start vertex of at least $t/6$ edges in $M$ of type $L$ (otherwise there would be fewer than $s/6 \cdot t+ s \cdot t/6=st/3$ edges in total). Given a directed edge $uw \in M$ of type $L,$ with $u$ being one of these $s/6$ leaves, we define $A_{uw}$ as the set of extending vertices of $uw$ belonging to $L$. We will now state some properties of these sets $A_{uw},$ which will allow us to find a big independent set. Since this is the most technical part of the proof and, once the appropriate properties are identified, is independent of the rest of the argument we prove it as a separate lemma afterwards. 

\textbf{1.)} No vertex belongs to more than $4$ different $A_{uw}$'s. Since $A_{uw} \subseteq L$ and by \Cref{cor:partition} part a), any vertex in $L$ belongs to at most $4$ different $v$-triangles, this means it can belong to at most $4$ different $A_{uw}$'s. \textbf{2.)} $|A_{uw}| \le x$. This follows since, by our definition, there are exactly $x$ $uw$-extending vertices.\\
\textbf{3.)} $\sum_{w}|A_{uw}| \ge tx/18$. This follows since for any $u$ there are at least $t/6$ edges $uw$ for which $A_{uw}$ is defined and each such edge being of type $L$ means there are at least $x/3$ extending vertices, meaning $|A_{uw}| \ge x/3$.\\ \textbf{4.)} If $uw$ and $u'w'$ are independent, then there can be no edges between $A_{uw}$ and $A_{u'w'}$. Else, we find $H_7$. 

This precisely establishes the conditions of \Cref{grid-lemma}, which provides us with an independent set of size $\min(\Omega(s\sqrt{tx}),\Omega(s^{1/2}t^{3/4}x^{1/4}), \Omega(s^{3/5}t^{3/5}x^{2/5}))$. Each of the three expressions is minimised when $t$ is as large as possible (under the assumption $st^2 \ge n^{1-o(1)}$) so we may plug in $t=n^{5/12}$ and $s=n^{2/12-o(1)}$ in which case the first expression evaluates to $n^{5.5/12-o(1)},$ the second to $n^{5.25/12-o(1)}$ and the third to $n^{5/12-o(1)}$.
    
\textbf{Third case:} there are at least $st/3$ edges in $M$ of type $I$.\\ 
Since $I$ spans an independent set we know $|I| < \alpha$. Any directed edge $uw$ in $M$ of type $I$ has at least $x/3$ extending neighbours in $I$. Since these are distinct for different $w$'s by our definition of an extending neighbour and since we insist that extending neighbours are out-neighbours of $u$ this means that overall there are at least $stx/9$ edges directed from leaves of $S$ to $I$.

This means that the average out-degree from $S$ to $I$ is at least $tx/9$. This together with a standard application of Cauchy-Schwarz implies there are $\Omega(s(tx)^2)=n^{16/12-o(1)}$ (recall that $st^2 \ge s_v/(2\log n) \ge n^{1-o(1)}$) out-directed cherries ($K_{1,2}$'s) with the centre in $S$. If we denote by $\mathcal{P}$ the set of pairs of vertices in $I,$ then $|\mathcal{P}|=\binom{|I|}{2} \le \alpha^2$, let us also denote by $d_p$ the number of common in-neighbours of a pair of vertices $p \in \mathcal{P}$. So in particular, $\sum_{p \in \mathcal{P}} d_p =\Omega(s(tx)^2)=n^{16/12-o(1)}$. Pairs $p$ with $d_p<2\alpha$ contribute at most $|\mathcal{P}|\cdot 2\alpha \le n^{15/12-o(1)}$ to this sum so if $\mathcal{P}'\subseteq \mathcal{P}$ denotes the set of pairs which have $d_p \ge 2\alpha,$ then also $\sum_{p \in \mathcal{P}'} d_p =\Omega(s(tx)^2)=n^{16/12-o(1)}$. Applying Tur\'an's theorem inside a common neighbourhood of $p \in \mathcal{P'}$ we find there $d_p^2/(4\alpha)$ edges (using that $d_p \ge 2\alpha$) or there is an independent set of size $\alpha$. Note that any edge we find inside this common in-neighbourhood gives rise to our desired $K_4$ minus an edge. In particular, using Cauchy-Schwarz we find at least $\sum_{p \in \mathcal{P}'} \frac{d_p^2}{4\alpha} \ge \frac{(\sum_{p \in \mathcal{P}'} d_p)^2}{4\alpha|\mathcal{P}'|} \ge \Omega(s^2(tx)^4/\alpha^3)\ge s_v^2/n^{7/12+o(1)}$ copies of our desired $K_4$ minus an edge, as claimed. 
\end{proof}

Recall that $2T_4 = \sum_{v} s_v$. Note also that stars with $s_v \ge T_4/n$ contribute at least $T_4$ to this sum. Hence, taking the sum over $v$ of the number of copies of our $K_4$'s minus an edge with the spine in $N(v)$ we obtain at least $\sum s_v^2/n^{7/12+o(1)} \ge T_4^2/n^{19/12+o(1)} \ge T_4 n^{5/12-o(1)},$ where we used Cauchy-Schwarz in the first inequality and our bound $T_4 \ge \Omega(n^{2})$, from the first claim in the second. Note however that certain copies of our $K_4$'s minus an edge got counted multiple times. But, for every $v$ that counted our $K_4$ minus an edge, we know it had its spine inside $N(v)$. This means that a single copy could be counted at most $\alpha(G)$ times since the spine (being an edge in $G$) can have at most $\alpha(G)$ neighbours, as they span an independent set. In particular, unless $\alpha(G) \ge n^{5/12-o(1)}$, this shows that there are more than $T_4$ distinct copies of our $K_4$'s minus an edge, contradicting the definition of $T_4$ and completing the proof.
\end{proof}

We now prove the lemma we used in the proof above. Let us first attempt to help the reader parse the statement. It says that if we can partition vertices of $G$ into a grid of subsets each of size at most $x$ (so each cell of the grid contains at most $x$ vertices), such that $G$ only has edges between vertices in the same row or column of the grid, and we additionally know that there is a large number of vertices in each row, then we can find a big independent set in the whole graph. 
\begin{lem}\label{grid-lemma}
Let $G$ be a triangle-free graph with vertex set $\bigcup_{i,j} A_{ij}$ where $i \in [s], j \in \mathbb{N}$. If
\vspace{-0.2cm}
\begin{enumerate}
    \item no vertex appears in more than $4$ different $A_{ij}$'s;
    \item $|A_{ij}| \le x,$ for any $i,j$;
    \item for some {$t\ge s$} and any $i \in [s]$ there are at least $tx$ vertices in $\cup_{j} A_{ij}$; 
    \item there are no edges of $G$ between $A_{ij}$ and $A_{k\ell}$ for any $i \neq k$ and $j \neq \ell,$
\end{enumerate}
\vspace{-0.2cm}
then $\alpha(G) \ge \min(\Omega(s\sqrt{tx}),\Omega(s^{1/2}t^{3/4}x^{1/4}), \Omega(s^{3/5}t^{3/5}x^{2/5}))$.
\end{lem}

\begin{proof}
Let us first replace any vertex which appears in multiple $A_{ij}$'s with distinct copies of itself, one per $A_{ij}$ it appears in. Our new graph has all $A_{ij}$ disjoint and satisfies the same conditions as the original. In addition, the independence number went up by at most a factor of $4$ so showing the result for our new graph implies it for the original. So let us assume that sets $A_{ij}$ actually partition the vertex set of $G$.

Let $\alpha=\alpha(G)$. We will call $\cup_{j} A_{ij}$ a row of our grid, $\cup_{i} A_{ij}$ a column and each $A_{ij}$ a cell. Let us first clean up the graph a bit. As long as we can find an independent set $I$ of size more than $2\alpha/s$ using vertices from at most $t/s$ cells inside some row, we take $I$, delete the rest of the row and all the columns containing a vertex of $I$ from $G$. If we repeat this at least $s/2$ many times we obtain an independent set of size larger than $\alpha$ which is impossible. This means that upon deleting at most $s/2$ many rows and at most $(t/s) \cdot s/2=t/2$ many columns we obtain a subgraph for which in any row any $t/s$ cells don't contain an independent set of size at least $2\alpha/s$. This subgraph still satisfies all the conditions of the lemma with $t:=t/2$ and $s:=s/2$. The only non-immediate condition is 3, it holds since we deleted at most $t/2$ cells in any of the remaining rows, so in total at most $tx/2$ vertices in that row altogether, using that any cell contains at most $x$ vertices. From now on we assume our graph $G$ satisfies the property that in any row any $t/s$ cells don't contain an independent set of size $2\alpha/s$ 

Let us delete vertices from our graph until we have exactly $tx$ in every row.
Let $n$ denote the number of vertices of $G$, so $n =stx$. Observe that at least half of the vertices of $G$ have degree at least $n/(4\alpha)$ as otherwise, vertices with degree lower than this induce a subgraph which has an independent set of size at least $\alpha$ by Tur\'an's theorem. Condition $4$ ensures each such vertex either has at least $n/(8\alpha)$ neighbours in its row or $n/(8\alpha)$ neighbours in its column. In particular, at least a quarter of vertices of $G$ fall under one of these cases. Since $G$ is triangle-free, the neighbourhood of any vertex is an independent set. We conclude that either there are at least $s/4$ rows containing an independent set of size $n/(8\alpha)$ or there are $t/4$ columns containing an independent set of size $n/(8\alpha)$.

Let us first consider the latter case. Let $U$ be the union of our $t/4$ independent sets of size at least $n/(8\alpha)$, belonging to distinct columns, so consisting of at least $tn/(32\alpha)$ vertices. Let $a_i$ denote the number of vertices of $U$ in row $i$. Then $\sum_{i=1}^{s} a_i \ge tn/(32\alpha)$ and each $a_i \le tx$ (since we removed all but $tx$ vertices in any row). On the other hand, since $G$ is triangle-free, we know that in each row we can find an independent subset of $U$ of size $\sqrt{a_i}$. In particular, since $U$ was constructed as a union of independent sets in columns (and all edges of $G$ are either within columns or within rows) this means that $U$ contains an independent set of size $\sum_{i=1}^{s} \sqrt{a_i} \ge \frac{tn}{32\alpha \cdot tx} \cdot \sqrt{tx}=\Omega(st^{3/2}x^{1/2}/\alpha)$ (where we used $n=stx$ and the standard fact that sum of roots is minimised, subject to constant sum, when as many terms as possible are as large as possible). In other words, we showed $\alpha \ge \Omega(st^{3/2}x^{1/2}/\alpha)$ giving us the second term of the minimum.

Moving to the former case let us again take a union $U$ of our $s/4$ independent sets of size at least $n/(8\alpha)$, belonging to distinct rows, so again $|U| \ge sn/(32\alpha)$. Call a cell $A_{ij}$ full if it contains at least $4\alpha/t$ vertices of $U$. There are fewer than $t/(2s)$ full cells in any row since otherwise, $U$ restricted to $\ceil{t/(2s)}$ full cells gives us an independent set of size at least $2\alpha/s$ using at most $\ceil{t/(2s)} \le t/s$ cells (using $t \ge s$), which contradicts our property from the beginning. Using this and once again the property from the beginning we conclude there can be at most $2\alpha/s$ vertices of $U$ in full cells of any fixed row. If $2\alpha/s\ge n/(16\alpha),$ then $\alpha^2 \ge \Omega (ns) \ge \Omega(s^2tx),$ so first term of the minimum is satisfied. So we may assume $2\alpha/s\le n/(16\alpha)$. Hence, by removing from $U$ any vertex belonging to a full cell we remove at most half the vertices of $U$ (since $U$ had at least $n/(8\alpha)$ vertices in every row). Now, finally, denote by $a_i$ the number of vertices of $U$ belonging to the column $i$. So $\sum_{i=1}^{s} a_i = |U| \ge sn/(64\alpha)$ and $a_i \le s \cdot 4\alpha/t$, since all the remaining vertices of $U$ belong to non-full cell. As before $\alpha \ge \sum_{i=1}^{s} \sqrt{a_i} \ge \frac{sn/(64\alpha)}{{4s\alpha/t}} \cdot \sqrt{4s\alpha/t}\ge \Omega(s^{3/2}t^{3/2}x/\alpha^{3/2})$ giving us the third term of the minimum.
\end{proof}

\section{2-density and local independence number}\label{sec:2-density}

In this section, we show our upper bounds on the minimum possible $\alpha(G)$ in a graph $G$ satisfying $\alpha_m(G)\ge r$. To do this we need to exhibit a graph with no large independent set in which any $m$-vertex subgraph contains an independent set of size $r$. As discussed in the introduction the natural candidates are random graphs and the answer is controlled by $M(m,r)$ which is defined to be the minimum value of the $2$-density over all graphs $H$ on $m$ vertices having $\alpha(H) \le r-1.$ It will be convenient to define $d_2(H)=\frac{e(H)-1}{|H|-2},$ so that the $2$-density is defined as the maximum of $d_2(H')$ over subgraphs $H'$ of order at least $3$, from now on whenever we consider $2$-density we will implicitly assume the subgraphs we take to have at least $3$ vertices. We begin by proving our reduction to the $2$-density Tur\'an problem, namely \Cref{prop:m2-loc-ind}.

\equiv*
\begin{proof}
A graph $H$ is said to be \textit{strictly $2$-balanced} if $m_2(H)>m_2(H')$ for any proper subgraph $H'$ of $H$. I.e., if $H$ itself is the maximiser of $d_2(H)$ among its subgraphs and in particular $m_2(H)=d_2(H)=\frac{e(H)-1}{|H|-2}$.

Let $\HH =\{H_1,\ldots, H_t\}$ be a collection of strictly $2$-balanced graphs such that any $m$-vertex $H$ with $\alpha(H) \le r-1$ contains some $H_i$ as a subgraph. We can trivially obtain it by replacing any $H$ in our family which is not strictly $2$-balanced by its subgraph $H'$ which maximises $d_2(H').$ In particular, $\HH$ is a family of strictly balanced $2$-graphs $H$ satisfying $m_2(H)=d_2(H)\ge M,$ with at most $m$ vertices and with the property that if a graph is $\HH$-free, then it satisfies $\alpha_m \ge r$. Note that $t\le 2^{\binom{m}{2}},$ so is in particular bounded by a constant (depending on $m$).

Let $G \sim \G(n,p),$ where we choose $p:=1/(48tn^{1/M})$ and will be assuming $n$ to be large enough throughout. Let $A_K$ denote the event that a subset $K \subseteq V(G)$, consisting of $k:=\frac{8\log n}{p}+2=O(n^{1/M}\log n)$ vertices, spans an independent set. In particular, we have $\P(A_k)=(1-p)^{\binom{k}{2}}$. Let $B_j^{i}$ denote the event that we find a copy of $H_i$ at the $j$-th possible location (so fixing the subset of vertices of $G$ where we could find $H_i$ and the labellings of vertices). In particular, $\P(B_j^i)=p^{e(H_i)}.$ Our goal is to show that with positive probability none of the events $A_K$ or $B_j^i$ occur, which implies that there is an $\HH$-free graph with no independent set of size $O(n^{1/M}\log n)$ as desired. We will do so by using the asymmetric version of the Lov\'asz local lemma (see Lemma 5.1.1 in \cite{alon-spencer}). To apply the lemma we first need to understand how many dependencies there are between different types of events. In particular, given $A_K$ it depends only on $\binom{k}{2}$ edges of $G,$ so in particular it is mutually independent of all events $B_{i}^j$ which do not contain one of these edges. In particular, it is mutually independent from all but at most $ k^2 n^{|H_i|-2}$ events $B_{i}^j$ and at most $n^k$ other events $A_{K'}$ (since there are at most this many such events in total). Similarly, any $B_j^i$ is mutually independent of all but at most $e(H_i)n^{|H_{i'}|-2}\le m^2n^{|H_{i'}|-2}$ events $B_j^{i'}$ for a fixed $i'$ and at most $n^k$ events $A_k$. 

We now need to choose parameters $x$ (corresponding to events of type $A_K$) and $y_i$ (corresponding to events of type $B_i^j$) such that $$\P(A_K) \le x\cdot (1-x)^{n^k} \cdot \prod_{i} (1-y_i)^{k^2n^{|H_i|-2}} \:\:\:\:\: \text{ and } \:\:\:\:\:  \P(B_i^j) \le y_i \cdot (1-x)^{n^k} \cdot \prod_{i'} (1-y_{i'})^{m^2n^{|H_{i'}|-2}}$$ which will complete the proof. We choose $x=1/n^k$ so that in particular $(1-x)^{n^k} \ge 1/3$ (as $n$ is large) and $y_i=p/(8tn^{|H_i|-2}) \le 1/2$ so that in particular $(1-y_i)^{n^{|H_i|-2}} \ge e^{-p/(4t)}$ (using $1-a \ge e^{-2a}$ for $a \le 1/2$). With these choices we obtain
$$x\cdot (1-x)^{n^k}\cdot\prod_{i} (1-y_i)^{k^2n^{|H_i|-2}} \ge \frac{1}{n^k} \cdot \frac13 \cdot e^{-pk^2/4}\ge  e^{-2k\log n-pk^2/4} = e^{-p\binom{k}{2}} \ge (1-p)^{\binom{k}{2}} = \P(A_K) \:\:\text{ and}$$ 
\vspace{-0.6cm}
\begin{align*} y_i \cdot (1-x)^{n^k} \cdot \prod_{i'} (1-y_{i'})^{m^2n^{|H_{i'}|-2}} & \ge \frac{p}{8tn^{|H_i|-2}} \cdot \frac{1}{3} \cdot e^{-m^2p/4} \ge \frac{p}{48t\cdot n^{\frac{|H_i|-2}{e(H_i)-1}\cdot (e(H_i)-1)}}  \ge \frac{p}{(48tn)^{\frac{1}{M}\cdot (e(H_i)-1)}} \\ & \ge p^{e(H_i)} \ge \P(B_i^j).
\end{align*}
Here, in the second inequality, we used the fact that $m$ is a constant while $p \to 0$ so $m^2p \to 0$ and in particular $e^{-m^2p/4} \le 1/2$ (since $n$ is large). In the third inequality we used $n^{-\frac{|H_i|-2}{e(H_i)-1}} \ge n^{-\frac{1}{M}},$ which follows since $M$ is equal to the minimum of $\frac{e(H_i)-1}{|H_i|-2}$ over $H_i$ (and we used $|H_i| \ge 3$ to put the $48t$ factor under the exponent).
\end{proof}

\textbf{Remark.} This result appears to be the best one can expect to get using random graphs, up to the polylog factor. The polylog factor can likely be slightly improved compared to the above argument by using the $\mathcal{H}$-free process (see e.g. \cite{osthus-h-free} for more details about this process). 

If we replace $m_2(H)$ in the definition of $M(m,r)$ with $d_2(H)=\frac{e(H)-1}{|H|-2}$ the problem of determining $M(m,r)$ would reduce to the classical Tur\'an's theorem. Indeed, since the number of vertices is fixed, minimising $d_2(H)$ is tantamount to minimising the number of edges in an $m$-vertex graph with $\alpha(H) <r$ and upon taking complements we reach the setting of the classical Tur\'an's theorem. This is why it is natural to call our problem of determining $M(m,r)$ the $2$-density Tur\'an problem. Note that since $m_2(H) \ge d_2(H)$ the proposition also holds if we replace $M$ with $\min d_2(H)$. This essentially recovers the argument of Linial and Rabinovich \cite{L-R}. However, it turns out one can in many cases do much better by using the actual $2$-density.

\subsection{The 2-density Tur\'an problem}\label{sec:turan-2density}
In this subsection, we show our results concerning the $2$-density Tur\'an problem of determining $M(m,r)$ which together with \Cref{prop:m2-loc-ind} give upper bounds in the local to global independence number problem mentioned in the introduction.

\subsubsection{Triangle-free case}\label{sec:triangle-free-ub}
Here we show our bounds for the case $k=3$. This means that $m$ and $r$ satisfy $2r-1\le m \le 3r-3$ and as expected the behaviour will be very different at the beginning and end of the range. Our first observation determines $M(2r-1,r)$.
\begin{prop}\label{prop-2r-1}
Let $r \ge 2$. Then $M(2r-1,r)=m_2(C_{2r-1})=1+\frac{1}{2r-3}.$
\end{prop}
\begin{proof}
Since $C_{2r-1}$ is a $2r-1$ vertex graph with no independent set of size $r$ we obtain $M(2r-1,r)\le m_2(C_{2r-1}).$ For the lower bound let $G$ be a graph on $2r-1$ vertices with $\alpha(G) \le r-1$, our goal is to show $m_2(G) \ge m_2(C_{2r-1})$. If $G$ contains a cycle of length $\ell,$ then $m_2(G) \ge m_2(C_\ell) \ge m_2(C_{2r-1})$ where in the last inequality we used $\ell \le 2r-1$ since $G$ has only $2r-1$ vertices. If $G$ contains no cycles it is a forest so in particular it is bipartite. One part of the bipartition must have at least $r$ vertices giving us an independent set of size at least $r$, which is a contradiction. 
\end{proof}

Turning to the other end of the range we show.
\begin{thm}
For $r \ge 2$ we have $M(3r-4,r)\ge \frac{5}{3}-\frac{1}{r-2}.$
\end{thm}
\begin{proof}
Let $G$ be a graph on $m=3r-4$ vertices with $\alpha(G) \le r-1$. If $G$ contains a triangle, then $m_2(G) \ge 2$ and we are done. If $G$ contains a subgraph $G'$ on $m'$ vertices with minimum degree at least $4$, then $m_2(G) \ge d_2(G') \ge \frac{2m'-1}{m'-2}>2$ so again we are done. In particular, we may assume that $G$ is $3$-degenerate. These conditions allow us to apply a modification of a result of Jones \cite{jones} (see \Cref{appendixC} for more details about the modification) which tells us that $e(G) \ge 6m-13(r-1)-1=5r-12$. This implies $m_2(G) \ge d_2(G) \ge \frac{5r-13}{3r-6}= \frac{5}{3}-\frac{1}{r-2}$ as claimed.
\end{proof}

This is close to the best possible, for example, the chain graph $H_r$ (see \cite{jones} for more details) has $3r-4$ vertices, no independent set of size $r$ and $m_2(H_r)=\frac{5}{3}-\frac{1}{9}\cdot \frac{1}{r-2}$. We believe that as in the problem of \cite{jones}, these graphs should be optimal, it is not hard to verify that this is indeed the case for the first few values of $r$ and one can improve our result above by repeating more carefully the stability type argument from \cite{jones} for our graphs.

In the above result, we did not look at the very end of the range for $k=3$, namely $m=3r-3$. The reason is that it seems to behave differently. Of course, $M(3r-3,r) \ge M(3r-4,r)$ so the same bound as above applies, however, it seems possible that a stronger bound is the actual truth, it is even possible that the answer jumps to $M(3r-3,r)\ge 2$.

\subsubsection{Independence number two.}\label{sec:independence-number-two}
In this subsection, we solve the $2$-density Tur\'an problem for graphs with no independent sets of size $3$. 
The behaviour depends on the parity of $m$, we begin with the easier case when $m$ is even.

\begin{lem}\label{thm:m2k3-lb}
For any $k \ge 2$ we have $M(2k,3) \ge (k+1)/2.$
\end{lem}
\begin{proof}
Let $G$ be a graph on $2k$ vertices with $\alpha(G)\le 2.$ This condition implies that for any vertex $v$ of $G$ the set of vertices not adjacent to $v$ must span a clique since otherwise, the missing edge together with $v$ makes an independent set in $G$ of size $3$. On the other hand, if we can find $K_k \subseteq G,$ then $m_2(G) \ge m_2(K_k)=\frac{k+1}{2}$ and we are done. So we may assume $G$ is $K_k$-free. Combining these two observations implies every vertex has at most $k-1$ non-neighbours and in particular $\delta(G) \ge 2k-1 -(k-1)=k$. This in turn implies $m_2(G) \ge \frac{e(G)-1}{|G|-2}\ge \frac{k^2-1}{2k-2}=(k+1)/2$ completing the proof.
\end{proof}

We now turn to the more involved case of odd $m=2k-1$. The increase in difficulty is partially due to the fact that the answer becomes very close (but not equal) to $m_2(K_k)$ which we have seen above is the answer for graphs with one more vertex. So the bound we need to show is much stronger in the odd case. We begin with the following lemma which is at the heart of our argument. We state it for the complement of our actual graphs for convenience.

\begin{lem}\label{lem:up-bip}
Let $k \ge 5$ and $1 \le t < \sqrt{(k-1)/2}$. Let $G$ be a triangle-free graph on $2k-1$ vertices with the property that any $k$ of its vertices span at least $t+1$ edges. Then $e(G) \le (k-1)^2-t^2+1$.
\end{lem}
\begin{proof}
The following easy claim will be used at various points in the proof. It also provides an illustration for the flavour of the more involved arguments we will be using later.

\begin{claim*}
If there are $2$ vertex disjoint independent sets of order $k-1,$ then $e(G) \le (k-1)^2-t^2+1.$
\end{claim*}

\begin{proof}
Let $v$ be the (only) vertex not belonging to either of the independent sets, which we call $L$ and $R$. Let $i$ denote the number of neighbours of $v$ in $L$ and $j$ in $R$. Observe first that $v \cup L$ and $v \cup R$ are both sets of $k$ vertices so need to span at least $t+1$ edges, by our main assumption on $G$. Since all edges in these sets are incident to $v$ ($L$ and $R$ are both independent sets) we conclude that $i,j \ge t+1$. Note further that, since $G$ is triangle-free, there can be no edges between neighbours of $v$, which means that there can be at most $(k-1)^2-ij$ edges in $L \cup R = G \setminus v$. Adding the $i+j$ edges incident to $v$ we obtain $e(G) \le i+j+(k-1)^2-ij=(k-1)^2+1-(i-1)(j-1) \le (k-1)^2+1-t^2.$
\end{proof}

We now proceed to obtain some information on the structure of $G$. Observe first that by our assumption on $t$ we have $(k-1)^2-t^2+1 > (k-1)^2-(k-1)/2+1=k^2-5k/2+5/2$ so if we can show $e(G) \le k^2-5k/2+5/2$ we are done. So let us assume $e(G) \ge k^2-5k/2+3$, which will suffice to give us some preliminary information about $G$. 

Since $G$ is triangle-free, neighbours of any vertex span an independent set. By our main assumption on $G$ there can be no independent set of order $k$ so $\Delta(G) \le k-1$. On the other hand, we have $\Delta(G)\ge 2e(G)/(2k-1) \ge (2k^2-5k+6)/(2k-1) > k-2,$  so $\Delta(G)=k-1.$ In particular, there exists a vertex with $k-1$ neighbours, which means that there is an independent set $R$ of size $k-1$ in $G$. If every vertex in $R$ has degree at most $k-3,$ then the sum of degrees in $G$ is at most $k(k-1)+(k-1)(k-3)=2k^2-5k+3 < 2e(G)$. So there is a vertex in $R$ with degree at least $k-2$ and in particular, there exists an independent set of size at least $k-2$ disjoint from $R$. In other words, $L:=G \setminus R$ contains an independent set of size $k-2$ and two remaining vertices, say $v$ and $u$ (see \Cref{fig:1} for the illustration of the current state). Let us w.l.o.g. assume that $v$ has at most as many neighbours in $L$ as $u$.
\begin{figure}
\begin{minipage}[t]{0.35\textwidth}
\centering
\captionsetup{width=\textwidth}
\begin{tikzpicture}[scale=0.98]
\defPt{0}{0}{l}
\defPt{3}{0}{r}
\defPt{0.4}{0.8}{d1}
\defPt{-0.4}{0.8}{d2}
\defPtm{($(l)+(d1)$)}{u}
\defPtm{($(l)+(d2)$)}{v}

\fitellipsis{$(l)+(0,1.2)$}{$(l)-(0,1.2)$}{0.8};
\fitellipsis{$(r)+(0,1.2)$}{$(r)-(0,1.2)$}{0.8};

\pic[rotate=45,scale=0.7] at ($(l)-(0,0.3)$) {K4};
\pic[rotate=18, scale=0.7] at ($(r)$) {K5};

\node[] at ($(u)+(0,-0.25)$) {$v$};
\node[] at ($(v)+(0,-0.25)$) {$u$};
\node[] at ($(l)+(-1,1.1)$) {$L$};
\node[] at ($(r)+(1,1.1)$) {$R$};

\draw[] (u) \smvx;
\draw[] (v) \smvx;
\end{tikzpicture}
\caption[First figure]{Initial structure for $k=6$, \\ dotted lines depict  missing edges.}
\label{fig:1}
\end{minipage}\hfill
\begin{minipage}[t]{0.6\textwidth}
\centering
\begin{tikzpicture}[yscale=0.7, xscale=1.2]
\defPt{0}{0}{l}
\defPt{3}{0}{r}
\defPt{0.4}{0.8}{d1}
\defPt{-0.4}{0.8}{d2}
\defPtm{($(l)+(d1)$)}{u}
\defPtm{($(l)+(d2)$)}{v}

\foreach \x in {1,...,5}{%
    \pgfmathparse{(\x-1)*360/5}
    \defPtm{(\pgfmathresult:2.2cm)}{v\x}
    }

\draw[] (v1) \smvx;

\defPt{-0.8}{0.15}{d1}
\defPt{0.7}{-0.4}{d2}
\defPtm{($(v2)+(d1)$)}{v21}
\defPtm{($(v2)+0.33*(d1)+0.66*(d2)$)}{v22}
\defPtm{($(v2)+0.66*(d1)+0.33*(d2)$)}{v23}
\defPtm{($(v2)+(d2)$)}{v24}

\draw[] (v21) \smvx;
\draw[] (v22) \smvx;
\draw[] (v23) \smvx;

\defPt{0.5}{0.65}{d1}
\defPt{-0.5}{-0.5}{d2}

\defPtm{($(v3)+(d1)$)}{v31}
\defPtm{($(v3)+0.33*(d1)+0.66*(d2)$)}{v32}
\defPtm{($(v3)+0.66*(d1)+0.33*(d2)$)}{v33}
\defPtm{($(v3)+(d2)$)}{v34}

\draw[] (v31) \smvx;
\draw[] (v32) \smvx;
\draw[] (v33) \smvx;

\foreach \i in {1,...,4}
{
\foreach \j in {1,...,4}
{
\draw[dashed] (v2\i) -- (v3\j);
}
}

\defPt{-0.8}{-0.15}{d1}
\defPt{0.7}{0.4}{d2}
\defPtm{($(v5)+(d1)$)}{v51}
\defPtm{($(v5)+0.33*(d1)+0.66*(d2)$)}{v52}
\defPtm{($(v5)+0.66*(d1)+0.33*(d2)$)}{v53}
\defPtm{($(v5)+(d2)$)}{v54}

\draw[] (v21) \smvx;
\draw[] (v22) \smvx;
\draw[] (v23) \smvx;

\defPt{0.5}{-0.65}{d1}
\defPt{-0.5}{0.5}{d2}

\defPtm{($(v4)+(d1)$)}{v41}
\defPtm{($(v4)+0.33*(d1)+0.66*(d2)$)}{v42}
\defPtm{($(v4)+0.66*(d1)+0.33*(d2)$)}{v43}
\defPtm{($(v4)+(d2)$)}{v44}

\draw[] (v31) \smvx;
\draw[] (v32) \smvx;
\draw[] (v33) \smvx;

\foreach \i in {1,...,4}
{
\foreach \j in {1,...,4}
{
\draw[dashed] (v4\i) -- (v5\j);
}
}

\foreach \i in {1,...,4}
{
\foreach \j in {1,...,4}
{
\draw[dashed] (v3\i) -- (v4\j);
}
}

\begin{scope}[rotate around={72:(0,0)}]
\draw[line width=1.5 pt] (v1) -- ($(v2)+(-0.3,-1.5)$);
\draw[line width=1.5 pt] (v1) -- ($(v2)+(0,1.1)$);
\draw[line width=1.5 pt] (v1) -- ($(v2)+0.7*(-0.3,-1.5)+0.3*(0,1.1)$);
\fitellipsis{$(v2)+(0,1.1)$}{$(v2)+(-0.3,-1.5)$}{0.7};
\end{scope}

\begin{scope}[rotate around={144:(0,0)}]
\fitellipsis{$(v3)+(0,1.1)$}{$(v3)+(0,-1.1)$}{0.7};
\end{scope}

\begin{scope}[rotate around={216:(0,0)}]
\fitellipsis{$(v4)+(0,1.1)$}{$(v4)+(0,-1.1)$}{0.7};
\end{scope}

\begin{scope}[rotate around={288:(0,0)}]
\draw[line width=1.5 pt] (v1) -- ($(v5)+(-0.3,1.5)$);
\draw[line width=1.5 pt] (v1) -- ($(v5)+(0,-1.1)$);
\draw[line width=1.5 pt] (v1) -- ($(v5)+0.7*(-0.3,1.5)+0.3*(0,-1.1)$);

\fitellipsis{$(v5)+(-0.3,1.5)$}{$(v5)+(0,-1.1)$}{0.7};
\end{scope}

\defPt{-1.8}{0.55}{v3}

\node[] at ($(v2)+(1.3,0.3)$) {$R_2$};
\node[] at ($(v1)+(0.25,0)$) {$v$};
\node[] at ($(v3)+(-0.8,1.5)$) {$L_1$};
\node[] at ($(v4)+(-0.8,-0.8)$) {$R_1$};
\node[] at ($(v5)+(1.3,-0.35)$) {$L_2$};

\draw[] ($(v3)+(0.45,0.5)$) \smvx;
\node[] at ($(v3)+(0.5,0.8)$) {$u$};

\draw[line width=1.5 pt] ($(v3)+(0.45,0.5)$) -- ($(v3)-(0,-1.3)$);
\draw[line width=1.5 pt] ($(v3)+(0.45,0.5)$) -- ($(v3)-(0.35,-0.7)$);
\draw[line width=1.5 pt] ($(v3)+(0.45,0.5)$) -- ($(v3)-(0.4,0)$);

\draw[line width=1.5 pt] ($(v3)+(0.45,0.5)$) -- ($(v5)+(0.6,0.5)$);
\draw[line width=1.5 pt] ($(v3)+(0.45,0.5)$) -- (v5);
\draw[line width=1.5 pt] ($(v3)+(0.45,0.5)$) -- ($(v5)+(0.3,0.25)$);

\end{tikzpicture}
\captionsetup{width=\textwidth}
\caption{Blow-up of $C_5$ structure, dashed lines denote only possible locations of edges, there can be no edges between non-adjacent parts or inside parts, apart from those belonging to $S$.
}
\label{fig:2}
\end{minipage}
\end{figure}

This almost gives us the situation in the claim above. In particular, by the claim, we may assume that both $u$ and $v$ have at least one neighbour in $L\setminus\{u,v\}$.  

We now proceed to obtain more detailed information on how $G$ should look like. 
Denote by $L_2:=N(v) \cap L, R_2:=N(v) \cap R, L_1=L \setminus (L_2 \cup \{v\})$ and $R_1:=R \setminus R_2$. Note that all edges within $L_1 \cup L_2$ must touch $u$ so $L_1 \cup L_2$ induces a star $S$ with a centre at $u$, say of size $s$. Note further that no edges in $R_2 \cup L_2=N(v)$ can exist as $G$ is triangle-free. Putting these observations together we conclude that $G$ without edges of $S$ is a subgraph of a blow-up of $C_5$ with parts $\{v\}, R_2, L_1, R_1$ and $L_2$ in order (see \Cref{fig:2} for an illustration in the case when $u \in L_1$).

This means that $v$ contributes $|L_2|+|R_2|$ edges while the remaining edges all come between $L_1 \cup L_2$ and $R_1 \cup R_2$ and $S$. Since $|R_1 \cup R_2|=|L_1 \cup L_2|=k-1$ and there can be no edges between $R_2$ and $L_2$ there are $(k-1)^2-|L_2||R_2|-X$ edges between $L_1 \cup L_2$ and $R_1 \cup R_2$ where $X$ denotes the number of non-edges between $R_2$ and $L_1$, $L_1$ and $R_1,$ and $R_1$ and $L_2$. In total we have $e(G) = |L_2|+|R_2|+(k-1)^2-|L_2||R_2|-X+s$.

Let us denote by $i=|R_2|$ and $j=|L_2|$ so $k-1-j=|L_1|;$ $k-1-i=|R_1|$ and $e(G)=i+j+(k-1)^2-ij-X+s=(k-1)^2+1-(i-1)(j-1)-X+s.$ Since by our main assumption on $G$ there needs to be at least $t+1$ edges among the $k$ vertices $v \cup R_1 \cup R_2$ and we know there are exactly $|R_2|=i$ edges in this set ($R_1 \cup R_2$ is an independent set) we conclude that $i \ge t+1.$ Similarly, we know $s+j \ge t+1$ as otherwise $v \cup L_1 \cup L_2$ make a $k$-vertex subset with $|L_2|+s \le t$ edges in total. Note that since we observed by the claim that both $v$ and $u$ need to have a neighbour inside $L \setminus \{u,v\}$ we must have $j,s \ge 1$, while by our choice of $v$ as having fewer neighbours in $L$ than $u$ we have $j\le s$.

We distinguish two cases depending on whether $v \sim u$ or not (i.e.\ whether $u \in L_1$ or $u \in L_2$). Let us deal with the case $u \in L_2$ first. There can be no edges within $L_2=N(v)$ so all leaves of $S$ must be in $L_1$. $u$ must have at least $t+1$ neighbours within $R_1 \cup R_2$ (otherwise $u \cup R_1 \cup R_2$ are $k$ vertices with at most $t$ edges) so there must be at least $s(t+1)$ edges missing between $R_1 \cup R_2$ and $L_1$, i.e. $X \ge s(t+1)$. In particular, $e(G) \le (k-1)^2+1-(i-1)(j-1)+s-s(t+1)\le (k-1)^2+1-t(j-1)-ts\le (k-1)^2+1-t^2,$ (where we used $j \ge 1$ and $i \ge t+1$ in the second inequality and $j-1+s\ge t$ in the third), as desired.

In the remaining case $u \in L_1.$ Let's say $s_1$ leaves of $S$ are in $L_1$ and $s_2$ in $L_2$. Let $x\ge t+1$ be the number of neighbours of $u$ in $R_1 \cup R_2$ we know as before there must be $xs_1+(k-1-x)$ non-edges between $L_1$ and $R_1 \cup R_2$ and at least $(x-|R_2|)s_2=(x-i)s_2$ non-edges between $R_1$ and $L_2$. In total we have $X \ge xs_1+(k-1-x) +\max(x-i,0)\cdot s_2=k-2+s+(x-1)(s-1)-\min (x,i)\cdot s_2.$ If we denote by $m:=\min(x,i)$ we get
\begin{align*}
    e(G)& \le (k-1)^2+1-(k-2)-(i-1)(j-1)-(x-1)(s-1)+m\cdot s_2\\
    & \le (k-1)^2+1-(k-2)-(m-1)(s+j-2)+m\cdot s_2\\
    & \le (k-1)^2+1-(k-2)-(m-1)(s+s_2-2)+m\cdot s_2\\
    & = (k-1)^2+1-(k-4)-(m-2)(s-2)-s_1
\end{align*}
Where we used $i,x \ge m$ in the second inequality and $j \ge s_2$ (since $s_2$ leaves live inside $L_2$ of size $j$) in the third. The term $(m-2)(s-2)+s_1$ is non-negative provided $s \ge 2$, since $m=\min(x,i)\ge t+1\ge 2.$ So if $s \ge 2,$ we have $e \le (k-1)^2+1-(k-4) \le (k-1)^2+1-t^2,$ where we used $t < \sqrt{(k-1)/2} \implies t^2 \le k-4$, which holds for $k \ge 5$ (using integrality of $t$ for $k=5,6$). If $s=1$ we must also have $j=1$ (since $j\le s$ and $j\ge 1$) and in turn $j+s\ge t+1$ implies $t=1$. If $s_2=0$ the first inequality above (and $k \ge 3$) gives $e(G) \le (k-1)^2=(k-1)^2-t+1$ and we are done. If $s_2=|L_2|=1,$ then $s_1=0$ and removing the single vertex in $L_2$ removes all neighbours of $v,u$ in $L$ and gives us again the situation from the claim. 
\end{proof}

We are now ready to deduce our bound on $M(2k-1,3)$.
\begin{thm}\label{thm:ub-odd}
Let $k \ge 4,$ then $M(2k-1,3) \ge \frac{k+1}2-\max\limits_{1\le t\le k-2}\min\left(\frac{t}{k-2},\frac{(k+1)/2-(t-1)^2}{2k-3}\right).$
\end{thm}
\begin{proof}
Let $G$ be a graph with $2k-1$ vertices and $\alpha(G) \le 2$. If $k=4$ the desired bound evaluates to $2$. Since in this case $G$ does not satisfy the \property{7}{3} by \Cref{lem:equivalence} it contains either a $K_4$ or an $H_7$ as a subgraph. Since $m_2(K_4)=5/2>2$ and $m_2(H_7) \ge \frac{e(H_7)-1}{|H_7|-2}=2$ (since $e(H_7)=11$ and $|H_7|=7$) we deduce that in either case $m_2(G) \ge 2$ as desired. Let us now assume $k \ge 5$.

Let $m(t)=\min\left(\frac{t}{k-2},\frac{(k+1)/2-(t-1)^2}{2k-3}\right)$. Observe that if we choose $t= \sqrt{(k-1)/2},$ we obtain $\frac{(k+1)/2-(t-1)^2}{2k-3}=\frac{t^2+1-(t-1)^2}{2k-3}=\frac{2t}{2k-3} < \frac{t}{k-2}.$ Let $a$ be the maximiser of $m(t).$  

If $m(a)=\frac{a}{k-2},$ then by the above observation we must have $a < \sqrt{(k-1)/2}$ (since the first term of the minimum is increasing and the second is decreasing in $t$). We may assume that among any $k$ vertices there are at least $a+1$ missing edges, as otherwise the induced subgraph on these $k$ vertices implies $m_2(G)\ge (k+1)/2-a/(k-2)=(k+1)/2-m(a)$ and we are done. This allows us to apply \Cref{lem:up-bip} to the complement of $G$ to deduce $G$ must have at least $\binom{2k-1}{2}-((k-1)^2+1-a^2)=k(k-1)+a^2-1$ edges. This in turn implies $m_2(G) \ge \frac{k(k-1)+a^2-2}{2k-3}=(k+1)/2-\frac{(k+1)/2-a^2}{2k-3} \ge (k+1)/2-m(a),$ where the last inequality follows since otherwise $m(a) < \frac{(k+1)/2-a^2}{2k-3}$ which implies $m(a)<m(a+1)$ (since also $\frac{a+1}{k-2} > \frac{a}{k-2}=m(a)$), which contradicts maximality of $m(a)$ (note that $a+1 <\sqrt{(k-1)/2}+1\le k-2$ for $k \ge 5$).

If on the other hand $m(a)=\frac{(k+1)/2-(a-1)^2}{2k-3}$ we may assume $a \ge 2$ (since $\frac{1}{k-2} < \frac{(k+1)/2}{2k-3}$ for $k \ge 5$). So we must have $m(a) \ge \frac{a-1}{k-2}$ (otherwise since the second term in the definition of $m(t)$ is decreasing in $t$ we conclude $m(a-1)> m(a)$ and get a contradiction). As in the previous case this means that any $k$ vertices must miss at least $a$ edges, or we are done. We also know that $m(a-1)=(a-1)/(k-2)$ (as otherwise again $m(a-1)>m(a)$) so again as in the previous case by our initial observation we must have $a-1<\sqrt{(k-1)/2}$ and we may apply \Cref{lem:up-bip} with $t=a-1$ to complement of $G$ to obtain $e(G)\ge k(k-1)+(a-1)^2-1$. This  implies 
\begin{equation} \label{ugly}
m_2(G) \ge \frac{k(k-1)+(a-1)^2-2}{2k-3}=\frac{k+1}2-\frac{(k+1)/2-(a-1)^2}{2k-3} = \frac{k+1}2-m(a),
\end{equation}
as desired.
\end{proof}

We now prove the lemma which we used for the $k=4$ case in the above theorem and mentioned in \Cref{sec:lower bounds}.
\begin{lem}\label{lem:equivalence}
Any $K_4$ and $H_7$-free graph $G$ satisfies $\alpha_7(G) \ge 3$.
\end{lem}
\begin{proof}
It is enough to show that any $7$ vertex graph $G$ which is $K_4$-free and has no $I_3$ (independent set of size $3$) must contain $H_7$. First observe that $\delta(G)\ge 3$ as otherwise there is a vertex with $4$ non-neighbours who must span a $K_4$, in order to avoid $I_3$. Note also that $\Delta(G) \le 5$ as if a vertex $v$ had degree $6,$ then by $R(3,3)=6$ in its neighbourhood we find an $I_3$ or a $K_3$ which together with $v$ makes a $K_4$. Since $G$ has odd size it must contain a vertex $v$ of degree exactly $4$. 

Our goal is to find two vertex disjoint triangles in $G$. If some vertex $v$ has degree $3,$ then its non-neighbours span a triangle and since its neighbours don't span an $I_3$ the edge among them together with $v$ give us our second triangle. If all vertices have degree at least $4,$ then by pigeonhole principle any two adjacent vertices lie in a triangle. If we take $v$ as our guaranteed vertex of degree $4,$ let $u,w$ be its non-neighbours. Then $u \sim w$ and $u$ and $w$ lie in some triangle. Removing it leaves us with $v$ and $3$ of its neighbours, so we again find a second triangle.

So we can always find a pair of vertex disjoint triangles $xyz$ and $abc$. Let $v$ be the remaining vertex of $G$. $v$ can send at most $2$ edges towards each of the triangles. If it sends exactly $2$ to both, then the non-neighbours of $v$ must be adjacent and we found our $H_7$. If it sends $2$ to $xyz$, say $v \sim x,y$ of them but only one to $abc$ say $v \sim a,$ then replacing $v$ with $z$ we find disjoint triangles $xyv$ and $abc$ such that $z$ sends $2$ edges towards each ($v \nsim z,b,c$ implies $z \sim b,c$) so we are back in the first case and are done. 
\end{proof}

Combining \Cref{thm:m2k3-lb} and \Cref{thm:ub-odd} with \Cref{prop:m2-loc-ind} we obtain \Cref{thm:main-ub-m-3}. Both above results are tight. Since we only needed the above lower bounds for \Cref{thm:main-ub-m-3} we will only describe our tightness examples here and postpone (the somewhat tedious) computation of their $2$-density to \Cref{appendix}. 

If $m=2k,$ then our example is simply a vertex disjoint union of $2$ cliques on $k$ vertices. This graph has no independent set of size $3$ and it is not hard to see that its $2$-density is equal to $m_2(K_k)=\frac{k+1}{2}$ (see \Cref{thm:ub-even}). If $m=2k-1$ the answer is more complicated since it needs to match the somewhat messy bound of \Cref{thm:ub-odd}. The examples however still arise naturally from looking at the proof and will be blow-ups of $C_5$ with cliques placed into parts which we choose to have sizes $1,a,k-1-a,k-1-a,a$ in order around the cycle, where $a$ is the optimal choice of $t$ in \Cref{thm:ub-odd}. The complement of any such graph is an actual blow-up of $C_5$ so is triangle-free and for the computation of its $2$-density see \Cref{lem:ub-odd-tight}.

\subsubsection{General Tur\'an 2-density problem}\label{sec:general-ub-2density}
In this section, we show our general bounds on $M(m,r)$. Combining the following proposition with \Cref{prop:m2-loc-ind} we obtain \Cref{general-ub}.  

\begin{prop}\label{prop-general-ub}
Let $k=\ceil{m/(r-1)}.$ Provided $m$ is sufficiently larger than $r$ we have $M(m,r) \ge \frac{k+1}{2}-\frac{c_r}{\sqrt{k}},$ where $c_r>0$ is a constant depending only on $r$.
\end{prop}
\begin{proof}
Let $G$ be an $m$-vertex graph with $\alpha(G)\le r-1,$ our task is to show $m_2(G)\ge \frac{k+1}{2}-\frac{c_r}{\sqrt{k}},$ where $k= \ceil{\frac{m}{r-1}}$, under the assumption that $m$ is large. $\alpha(G)\le r-1$ implies that $G$'s complement $\overline{G}$ is $K_r$-free. 
Let $t=t_{r-1}(m)-e(\overline{G}),$ where $t_{r-1}(m)$ denotes the Tur\'an number for $K_r$-free graphs on $m$ vertices. If $t \ge \frac32m,$ then we get $$m_2(G) \ge \frac{e(G)-1}{|G|-2}=\frac{\binom{m}{2}-t_{r-1}(m)+t}{m-2}\ge  \frac{\frac{m^2}{2(r-1)}+t-m/2}{m}\ge \frac{m}{2(r-1)}+1\ge \frac12 \cdot (k-1)+1 = \frac{k+1}{2},$$
where in the second inequality we used the standard bound $t_{r-1}(m) \le \left(1-\frac1{r-1}\right)\frac{m^2}2$.

So we are done unless $t<\frac32 m$. Since $\overline{G}$ is $K_r$-free and has $t_{r-1}(m)-t$ edges a stability theorem (see Theorem 1.3 in \cite{balogh2019making}) implies that $\overline{G}$ can be made $r-1$-partite by removing at most $\frac{rt^{3/2}}{2m}$ edges (being crude and using that $m$ is sufficiently larger than $r$). Translating this to $G$ we conclude $G$ is a vertex disjoint union of $r-1$ cliques missing a few edges, in total at most $\frac{rt^{3/2}}{2m} \le r\sqrt{m}$ edges. At least one of these ``cliques'' needs to have size $s \ge k$. In particular, if we take a subset of $k$ vertices of this ``clique,'' it still misses at most $r\sqrt{m}$ edges. In particular, it has $2$-density at least $\frac{k+1}{2}-\frac{r\sqrt{m}}{k-2}\ge \frac{k+1}{2}-\frac{c_r}{\sqrt{k}}$.
\end{proof}

\textbf{Remark.} The stability result we used above was also independently discovered in \cite{R-S} (in an asymptotic form), we used the variant from \cite{balogh2019making} since it is explicit. Our problem seems to be closely related to this type of stability problems for Tur\'an's theorem. For example, the bipartite variant, which was precisely solved in \cite{erdos-bip}, has the same form of optimal examples as we found for $M(2k-1,3)$. This was recently generalised to $r$-partite graphs in \cite{dano} which might be helpful for studying $M(m,r)$ for larger $r$.

Note that for the special case of $r=3$ and $m$ odd this result matches (up to a constant factor in front of the lower order term) our bound in \Cref{thm:ub-odd} and is hence almost best possible in this case by \Cref{lem:ub-odd-tight}. On the other hand, if $m$ is even it is some way off. This seems to happen in general, we found examples (disjoint unions of our examples for the $r=3$ case) which show that the above bound is tight up to the constant factor in front of the lower order term provided $m \pmod {r-1}$ is between $1$ and $(r-1)/2$. This condition ensures that in the Tur\'an $K_r$-free graph on $m$ vertices there are more small parts (of size $k-1$) which allows us to pair up small and big parts and place there a copy of our example from the $r=3$ case, we once again relegate the details to \Cref{appendix}. It seems that as $m \pmod {r-1}$ approaches $r$ stronger bounds should hold and ultimately if $m \mid (r-1),$ the lower order term disappears completely as it did in the $r=3$ case.

\Cref{general-ub}, while being close to best possible, unfortunately requires $m$ to be somewhat large (compared to $r$) which misses many interesting instances of the problem. The following result illustrates some of our ideas for obtaining results which hold for any choice of parameters. We restrict attention to the divisible case $m=k(r-1)$ to keep the argument as simple as possible.

\begin{prop}\label{prop:non-asymptotic-ub}
Let $r \ge 3$ and $m=k(r-1),$ then we have $M(m,r) \ge \frac{k}{2}+\frac{k-1}{m-2}$.
\end{prop}
\begin{proof}
We will prove the result by induction on $r$ while keeping $k$ fixed.
For the base case of $r=3$, we have $k=m/2$ and the statement matches precisely \Cref{thm:m2k3-lb}. 

Let $G$ be a graph on $m$ vertices having $\alpha(G) \le r-1$. If $G$ has a vertex $v$ of degree less than $k,$ then removing $v \cup N(v)$ from $G$ we obtain a graph $G'$ on at least $m-k=k(r-2)$ vertices which has $\alpha(G') \le r-2$, since $v$ extends any independent set we can find in $G'$. This implies by the inductive assumption for $r-1$ that $m_2(G) \ge m_2(G') \ge M(m-k,r-1) \ge \frac{k}{2}+\frac{k-1}{m-k-2}\ge \frac{k}{2}+\frac{k-1}{m-2}$ with room to spare. Hence, we may assume $\delta(G) \ge k$ which implies $m_2(G) \ge d_2(G) \ge \frac{mk/2-1}{m-2}=\frac{k}{2}+\frac{k-1}{m-2}$.
\end{proof}

The above proof clearly leaves quite some room for improvement. However, it (combined with $M(m,r)$ being increasing in $m$ to capture the non-divisible cases) already suffices to improve the bound of Linial and Rabinovich for all values of $m$ and $r$ with $k \ge 4$. It also suffices to obtain a significant improvement in the benchmark case $m=20,r=5$ of $M(20,5)\ge 49/18$ over the previously best bound of $43/18$ of Kostochka and Yancey \cite{kost}. We have more involved ideas which allow one to improve on the above bound quite substantially. In particular, we manage to resolve the benchmark case and show $M(20,5)=3$. Since that argument is somewhat more involved and its generalisations become even more complicated, while ultimately still falling short of the asymptotic result of \Cref{general-ub}, we relegate it to \Cref{appendixB}.

\vspace{-0.5cm}
\section{Concluding remarks and open problems}\label{sec:conc-remarks}
\vspace{-0.3cm}
In this paper we study the local to global independence number problem, i.e.\ how big an independent set one finds in a graph with the property that any $m$ vertices contain an independent set of size $r$. While many of our results break previous barriers on this problem, there is still room for improvement and we believe we have not fully exhausted the potential of our ideas.

In terms of lower bounds, we improve previously best bounds for about half of the possible choices of $m$ and $r$. It would be interesting to obtain a similar improvement for the whole, or at least most of the range. Our argument here relied on improving the bounds for $r=3$ which is then generalised through \Cref{lem:reduction}. One can follow our approach for $r \ge 4$ as well. \Cref{lem:reduction} easily generalises so for example if one improves bounds say for $m=3k-1, r=4,$ this leads to improvement for about $2/3$ of the possible values in general. For $r=3$ our arguments relied on a Ramsey result for graphs $H_{2k-1}$ which were certain blow-ups of $C_5$.  We believe a similar story should happen for larger $r$, in the initial cases role of $C_5$ could be taken by the chain graphs (see graphs $H_k$ in \cite{jones}) and to obtain a general result for fixed $r$ one should prove a Ramsey bound for appropriate blow-ups of these chain graphs. This should lead to an improvement for essentially all values of $m$ and $r$ except when $r-1 \mid m$ which seems more difficult. In fact, using a minor modification of \Cref{lem:reduction} if one improves the bounds in such a ``divisible'' case, say $m=2k$ and $r=3,$ this immediately improves the bounds for any choice of $m$ and $r$ with $r-1 \nmid m$ (and most divisible cases as well). Here a good starting point seems to be the case $m=8, r=3$.
\begin{qn}
Does any graph with an independent set of size $3$ among any $8$ vertices have $\alpha(G)\ge n^{1/3+\eps}$?
\end{qn}
The reason we raise the $(8,3)$ case instead of $(6,3)$ is that the latter is easily seen to be essentially equivalent to the problem of how large independent sets we find in triangle-free graphs. Since the answer to this classical problem is known up to a constant factor \cite{shearer, bohman-keevash,morris} the same holds for our problem.  
This raises the possibility that the $(8,3)$ case is essentially equivalent to the same problem for $K_4$-free graphs which is open and believed hard. It turns out however that this is not the case since for example the square of $C_8$ is an $8$ vertex $K_4$-free graph with no independent set of size $3$ and could play the role of our $H_{2k-1}$'s as the intermediate forbidden graph in this case. In fact, it is the only possible candidate, as can be seen by looking at optimal examples for $R(4,3)$ \cite{ramsey-examples} which show that property $\alpha_8 \ge 3$ is essentially equivalent to the graph being $K_4$-free and $C_8^2$-free.

In terms of improving our new bounds a good starting place are graphs in which every $7$ vertices have an independent set of size $3$. We showed such graphs must have $\alpha(G)\ge n^{5/12-o(1)}$ proving a conjecture of Erd\H{o}s and Hajnal. Here the natural limit for our methods is actually $n^{3/7}$ and most of our argument works up to this point. It should be possible to push our methods at least beyond $5/12$. On the other hand, breaking $3/7$ seems to require new ideas. The main question here is whether it is possible to reach $1/2$, namely whether the second conjecture of Erd\H{o}s and Hajnal holds.

\begin{qn}
Does any graph with an independent set of size $3$ among any $7$ vertices have $\alpha(G)\ge n^{1/2-o(1)}$?
\end{qn}
\vspace{-0.3cm}
\Cref{lem:equivalence} shows that this is in some sense equivalent to a Ramsey problem of our graph $H_7$ vs an independent set, with the added benefit that we know the graph is $K_4$-free, which however seems to be a weaker condition than being $H_7$-free, so it is unclear if it is actually needed at all. The above bound for $m = 7, r=3$ is stronger than our general bound which makes it likely that the general bound can be further improved.

Let us now turn to the upper bounds. Our bounds all arise from our results on the Tur\'an $2$-density problem and the main open problem is to solve this problem precisely for all choices of parameters.

\begin{qn}
What is the minimum value of the $2$-density of a graph on $m$ vertices having no independent set of size $r$?
\end{qn}


We defined the answer to be $M(m,r)$ and determine it precisely for $r=3$, for ends of the range with $k=3$ (for $m=2r-1$ and up to lower order term for $m=3r-4$), for certain small cases such as $m=20,r=5$ and determine it up to $O_r(1/\sqrt{m})$ in general. While the parameter $k=\ceil{\frac{m}{r-1}}$ seems to control the rough behaviour of $M(m,r)$, in order to obtain precise results one needs to take into account the residue of $m$ modulo $r-1$. We have seen this in the $r=3$ case with the distinction between even and odd cases. This is also evident in the $k=3$ case from our results for the ends of this range. The behaviour for $k=3$ across the whole range also seems interesting and may be a good starting point for obtaining precise general results. 

In general, we can show that for the first half of the non-zero residues the $O_r(1/\sqrt{m})$ term is needed. It could be interesting to determine what happens for the remaining half of the residues and in particular when $r-1 \mid m$. Based on our results for $r=3$ and $m=20,r=5$, it seems plausible that $M(m,r)=m_2(K_k)=\frac{k+1}{2}$ for any $r$ and $m=k(r-1)$.

Finally, let us summarise the current state of the art in the following table, where we remind the reader that $f(n,m,r)$ stands for the smallest possible size of $\alpha(G)$ in an $n$ vertex graph with $\alpha_m(G) \ge r$. 

\vspace{-0.4cm}
\begin{center}
\resizebox{\textwidth}{!}{
\tabulinesep=1.2mm
\begin{tabu}{c|c|c|c|c|c|c}
    $m$  & $2r-1$ & $ [2r, 3r-5]$ & $3r-4, 3r-3$& $\Big((k-1)(r-1),(k-\frac12)(r-1)\Big]$ & $\Big(k-\frac12)(r-1),k(r-1)\Big]$ \\
    \hline
     $f(n,m,r) \ge $ & $\Omega\left(n^{1-\frac{1}{r}}\right)$ &\multicolumn{2}{c|}{$\Omega\left(n^{1-{1}/{\left\lfloor\frac{m-r+1}{m-2r+2}\right\rfloor}}\right)$} &  $\Omega\left(n^{\frac{1}{k-3/2}}\right)$ & $\Omega\left(n^{\frac{1}{k-1}}\right)$\\ 
    \hline
    $f(n,m,r) \le $ & \multicolumn{2}{c|}{$O\left(n^{1-\frac{1}{2r-2}}\right)$} & $n^{\frac{3}{5}+\frac{2}{5r-13}+o(1)}$ & \multicolumn{2}{c|}{$n^{\frac{2+o(1)}{k+1-O_r({1}/{\sqrt{k}})}}$}   
\end{tabu}
}
\end{center}
Where all the asymptotics are in terms of $n\to \infty$ apart from the $O_r$ term where 
$r$ is a constant and $k \to \infty$. 
 
We note in addition that the value for $r\le m \le 2r-2$ is known precisely and is equal to $n-o(n)$. One can use our arguments to obtain improvements for various parts of the above regimes and special cases, but let us mention here specifically that we have slightly weaker upper bounds in the final regime which do not require $k$ to be large compared to $r$ (see \Cref{prop:non-asymptotic-ub}).

\textbf{Acknowledgments.} We would like to thank Alexandr Kostochka for helping us find \cite{jones} and Michael Krivelevich for useful comments. We would also like to thank the anonymous referees for their careful reading of the paper and numerous suggestions which improved the presentation of this paper. The first author would like to gratefully acknowledge the support of the Oswald Veblen Fund.



\providecommand{\bysame}{\leavevmode\hbox to3em{\hrulefill}\thinspace}
\providecommand{\MR}{\relax\ifhmode\unskip\space\fi MR }
\providecommand{\MRhref}[2]{%
  \href{http://www.ams.org/mathscinet-getitem?mr=#1}{#2}
}
\providecommand{\href}[2]{#2}

\appendix
\clearpage
\pagenumbering{roman}

\section{Independence number in sparse triangle-free graphs}
\label{appendixC}

\begin{thm}
Any $3$-degenerate, triangle-free graph $G$ with $m$ vertices and no independent set of size $r$ has $e(G) \ge 6m-13r-1$.
\end{thm}

\begin{proof}
Our proof is by induction on $m$. We will actually prove a slightly stronger result. We will show that $e(G) \ge 6m-13r-1+I$, where $I$ is equal to $1$ if $\delta(G)<3$ and equal to $0$ otherwise. 

For the base case we show the result for $m\le 5$ and any $r$. If $m \le 2,$ the result is immediate since the desired bound follows from $e(G) \ge 0.$ If $m\ge 3,$ since the graph is triangle-free we must have $r \ge 2,$ so the bound again follows for $m\le 4$ immediately and if $r\ge 3$ also for $m=5$. In the only remaining case, if $r=2,m=5,$ we must have $G=C_5$ and the bound again holds.

Now let us assume $m \ge 6$ and that the result holds for any graph on at most $m-1$ vertices satisfying our conditions. If there is an isolated vertex in $G,$ then by removing it we obtain a graph on $m-1$ vertices with independence number at most $r-1$ which implies $e(G) \ge 6(m-1)-13(r-1)-1 \ge 6m-13r+6\ge 6m-13r-1+I$. Similarly if there is a vertex of degree $1,$ removing it and its neighbour leaves us with a graph on $m-2$ vertices with independence number at most $r-1$ giving us the bound $e(G) \ge 6(m-2)-13(r-1)-1 \ge 6m-13r\ge 6m-13r-1+I$. So we may assume $\delta(G) \ge 2$.

If $\delta(G)=2$ let $v$ be a vertex of degree $2$, with neighbours $u,w$. If there are at most $2$ vertices, other than $v$, adjacent to one of $u$ or $w,$ then removing them and $v,u,w$ leaves us with a graph on at least $m-5$ vertices with independence number at most $r-2$. This leftover graph has at least $6(m-5)-13(r-2)-1=6m-13r-5$ edges. On the other hand $G$ in addition has at least $5$ edges touching these $5$ removed vertices since minimum degree is $2$ giving us $e(G) \ge 6m-13r \ge 6m-13r+I-1$ as desired. So there are at least $3$ vertices (other than $v$) adjacent to $u$ or $w$ and in particular $v,u,w$ are incident to at least $5$ edges. By removing $v,u$ and $w$ we obtain a graph on $m-3$ vertices with independence number at most $r-1$ which hence has at least $6(m-3)-13(r-1)-1=6m-13r-6$ edges. Since $v,u,w$ are incident to at least $5$ edges we need to gain one more. Note first that $G\setminus{u,v,w}$ must have minimum degree $3$ or we gain one from its $I$ term. Note also that either $u$ or $w$ must have degree $2$ or $v,u,w$ touch at least $6$ edges and we gain. Say $u$ is of degree $2$, and $w'$ is its neighbour other than $w$. $w'$ must have degree at least $3$ in $G\setminus{u,v,w}$ so at least $4$ in $G$. This means that $u,v,w'$ touch at least $6$ edges and repeating the above argument with $u$ in place of $v$ we are done.

Final case is if $\delta(G)=3$. We may assume $G$ is connected, as otherwise we may apply induction on each of the components and are done. If $G$ is $3$-regular, then the number of edges is $e(G)=3m/2$ and a result of Staton (see Theorem 6 in \cite{staton}) on graphs with maximum degree $3$ implies $r \ge \frac{5}{14} m$ this implies $e(G)=3m/2 \ge 6m-13r$ as desired. Hence, we may assume there is a vertex of degree at least $4$ and in particular, since $G$ is $3$-degenerate that there exists a vertex $v$ of degree $3$ adjacent to a vertex $u$ of degree at least $4$. Let $w,q$ be the remaining neighbours of $v$. Since $\delta(G)=3$ and $d(u)=4$ we know there are at least $10$ edges touching $v,u,w$ or $q$. Removing these $4$ vertices we get a graph on $m-4$ vertices with independence number at most $r-1$ so by induction it has at least $6(m-4)-13(r-1)+I'-1=6m-13r-1-11+I'$ edges, where $I'=1$ if the remainder graph has minimum degree at most $2$. This means that we are done unless there are exactly $10$ edges touching $v,u,w,q$ and the remainder graph has minimum degree at least $3$. $w$ is a vertex of degree $3$ with two neighbours in the remainder graph, so each having degree at least $3$ there and at least $4$ in $G$ (since they are adjacent to $w$). This means that $w$ and its neighbours touch at least $11$ edges so repeating the argument as above with $w$ in place of $v$ we obtain the desired bound.
\end{proof}

\section{The M(20,5) case}
\label{appendixB}

In order to show $M(20,5)\ge 3$, we will need a few intermediate results. Let us define $e(m,r)$ to be the minimum possible number of edges in an $m$-vertex graph with independence number at most $r-1,$ provided it has $2$-density less than $\frac{k+1}{2}$ (where $k=\ceil{m/(r-1)}$). It can be thought of as the variant of determining $M(m,r)$ where we care about the final number of edges instead of the $2$-density, but we impose a restriction of not having too dense parts. For determining $M(20,5)$ we will need bounds on $e(14,4)$ and $M(15,4)$ for which in turn we will need $e(9,3)$.

\begin{lem}\label{lem9-4}
$e(9,3) \ge 19$.
\end{lem}
\begin{proof}
Let $G$ be a graph with $9$-vertices with $\alpha(G) \le 2$ and $m_2(G)<3$. If $G$ has a vertex $v$ of degree at most $3,$ then $v$ has at least $5$ non-neighbours which must span a clique (since $\alpha(G) \le 2$) which implies $m_2(G) \ge m_2(K_5)=3$, a contradiction. If $\delta(G) \ge 5,$ then $e(G) \ge 9 \cdot 5/2 >19,$ and we are done. So $\delta(G)=4$ and there exists a vertex $v$ with degree exactly $4$. Let $L=v \cup N(v)$ and $R=G \setminus { L}$, so $|L|=5,|R|=4$. Since $|L|=5$ there must be a missing edge in $L$ (or we find a $K_5$ and have $m_2(G) \ge 3$). 
Observe that vertices making a missing edge in $L$ can not both be non-adjacent to the same vertex in $R$ (or $\alpha(G) \ge 3$) meaning they need to send at least $4$ edges towards $R$. We obtain that there must be at least $4+3\cdot 4-3=13$ edges touching $L$, since there are $4$ cross edges touching the missing edge, the remaining $3$ vertices inside $L$ each have degree at least $\delta(G)=4$ and we double counted only the edges between these $3$ vertices, so at most $3$. Now since $R$ consists of non-neighbours of $v$ it spans a $K_4$ so there are $6$ edges within $R$ and in total we have the claimed $13+6=19$ edges.
\end{proof}

\begin{lem}\label{lem14-4}
$e(14,4) \ge 33$.
\end{lem}
\begin{proof}
Let $G$ be a graph with $14$-vertices with $\alpha(G) \le 3$ and $m_2(G)<3$. If $G$ has a vertex $v$ of degree at most $3,$ then $v$ has at least $10$ non-neighbours which contain no independent set of size $2$ so must have $2$-density at least $3$ by \Cref{thm:m2k3-lb}. If $\delta(G) \ge 5,$ then $e(G) \ge 14 \cdot 5/2 >33$ as desired. So $\delta(G)=4$ and there exists a vertex $v$ with degree exactly $4$. Let again $L=v \cup N(v)$ and $R$ be the rest of the graph. Note now that a missing edge in $L$ must touch at least $5$ edges going to $R$, as otherwise there are $9-4\ge 5$ common non-neighbours of the missing edge and they must span a clique. Similarly as in \Cref{lem9-4} this means there needs to be $14$ edges touching $L$. On the other hand $G[R]$ is a $9$-vertex graph with $\alpha(G[R])\le 3$ (since $v$ is a non neighbour to anyone in $R$) and it has $m_2(G[r])\le m_2(G)<3$ so \Cref{lem9-4} applies implying there are at least $19$ edges inside $G[R]$ and giving us our claimed total.
\end{proof}

Backtracking along the above proofs it is not hard to show that they are both optimal and even deduce a lot about the structure of the optimal examples. We are now ready to prove our final intermediate result which determines $M(15,4)$. 

\begin{lem}\label{lem15-4}
$M(15,4) = 3$.
\end{lem}
\begin{proof}
Let $G$ be a graph with $15$-vertices with $\alpha(G) \le 3$ and $m_2(G)<3$. If there is a vertex of degree at most $4$ removing it and its neighbourhood, similarly as above reduces to the case of $M(10,3)$ which by \Cref{thm:m2k3-lb} is equal to $3$. So $\delta(G) \ge 5$. If $\delta(G) \ge 6,$ then $e(G) \ge 45$ and $d_2(G)\ge \frac{45-1}{15-2}>3$. So let $v$ be a vertex with degree exactly $5$. Let $L=v \cup N(v)$ and $R$ be the rest of the graph. $G[R]$ is a $9$-vertex graph with $\alpha(G[R])\le 2$ and it has $m_2(G[r])\le m_2(G)<3$, so $R$ spans at least $19$ edges by \Cref{lem9-4}. We claim there needs to be at least $21$ edges touching $L$. If there were a triangle of missing edges inside $L$ its vertices need to send at least $9$ edges across (or we get $\alpha(G) \ge 4$) so our usual calculation tells us there are $9+3\cdot 5-3=21$ edges touching $L$ as desired. Note also that missing edges inside $L$ can not span a star (or we can remove its centre and be left with a $K_5$ inside $L$ so there need to exist $2$ disjoint missing edges. Each missing edge sends at least $5$ edges across (or we find a $K_5$ in their common non-neighbourhood) and since we are assuming there is no missing triangle inside $L$ there needs to be $2$ actual edges inside $L$ between our missing pair. Putting these together we get $2\cdot 5+2+2\cdot 5-1=21$ edges touching $L$ as desired. So there are at least $21+19=40$ edges in $G$ and $m_2(G) \ge 3$ as desired.
\end{proof}

\begin{thm}
$M(20,5)=3.$
\end{thm}
\begin{proof}
Let $G$ be a graph with $20$-vertices with $\alpha(G) \le 4$ and $m_2(G)<3$. If there is a vertex of degree at most $4$ removing it and its neighbourhood, reduces to the case of $M(15,4)$ solved in \Cref{lem15-4}. So $\delta(G) \ge 5$. If there are more than $8$ vertices of degree greater than $5,$ then $e(G) \ge 55$ and $d_2(G)\ge \frac{55-1}{20-2}=3$. So there are at least $12$ vertices of degree $5$. Let $v$ be such a vertex and Let $L=v \cup N(v)$ and $R$ be the rest of the graph. $G[R]$ is a $14$-vertex graph with $\alpha(G[R])\le 3$ and it has $m_2(G[r])\le m_2(G)<3$, so $R$ spans at least $33$ edges by \Cref{lem14-4}. If we find $22$ edges touching $L$ we obtain  $e(G) \ge 55$ and are done.

Note that any missing triangle in $L$ sends at least $10$ edges across (or the $\ge 5$ common non-neighbours need to make a clique). So as in the previous lemma we obtain at least $22$ edges touching $L$ as desired. Once again the missing edges can not span a star or we find a $K_5$ inside $L$ so again we need to be able to find a disjoint pair of missing edges in $L$. Since again any missing edge must send at least $5$ edges to $R$ (or we reduce to the case of $M(10,3)=3$) and at least $2$ actual edges must exist between the missing edges (or we find a missing triangle) so again we obtain $2\cdot 5+2+2\cdot 5-1=21$ edges touching $L$. 

This already implies $M(20,5) \ge 53/18$ but to obtain our best bound we need to work a little bit harder. In particular, if we do find only $21$ edges touching $L$ above we obtain a lot of structural information about $L$.  

\begin{claim*}
If there are $21$ edges touching $L,$ then $G[L]$ induces either:
\begin{enumerate}
    \item[a)] A $K_3$ and a $K_4$ intersecting in a single vertex and consisting entirely of vertices of degree $5$ in $G$ or 
    \item[b)] Two $K_4$'s intersecting in an edge, one $K_4$ consists of vertices of degree $5$ and the remaining $2$ vertices are of degree $6$.
\end{enumerate}
\end{claim*}
\begin{proof}
We have already observed that there are no missing triangles in $L$ and that we can find $2$ disjoint missing edges, which gave us at least $21$ edges touching $L$. In order for this bound to be tight there needs to be exactly two actual edges in between them and in order to avoid missing triangles this means that the $4$ vertices making these two edges span a missing $C_4$ with both cross edges present. Each of the $4$ missing edges making this $C_4$ must send exactly $5$ edges across (or we again gain) and every vertex of this $C_4$ has at most $3$ neighbours in $L$ so has at least $2$ neighbours outside. The only way this can happen is if $2$ diagonal (so adjacent in $G$) vertices of $L$ send $3$ edges out (call them $w_1,w_2$) and remaining $2$ send $2$ edges out (call them $s_1,s_2$). By minimum degree in $G$ being at least $5$ we know that both $s_i$ must be joined to both remaining vertices in $L$ ($v$ and call the final vertex $u$). Since $v$ is adjacent to everyone the only edges we don't have information about are $uw_1,uw_2$. If only $uw_1$ is a missing edge, then $u,w_1$ must send at least $5$ edges across and since $w_1$ sends $3$ $u$ must send $2$. But this would imply $u$ is adjacent to $v,w_2,s_1,s_2$ and these $2$ outer vertices so have degree at least $6$ and improve our bound. Therefore, there are only $2$ options both $uw_1$ and $uw_2$ are edges, giving us the case b) or neither are edges, giving us case a). Note that we know $u,v$ must be of degree $5$ or we gain and since we know how many edges $s_i$'s and $w_i$'s send outside and how $G[L]$ looks like we know degree's of everyone in $G$.
\end{proof}

Observe that $v$ was an arbitrary vertex of degree $5$ so the above must hold for any such vertex.

\begin{claim*}
Case b$)$ happens for all vertices of degree $5$.
\end{claim*}
\begin{proof}
Let us assume case a) happens for vertex $v$. In other words $v \cup N(v)$ induces a triangle $v,u,w$ and a $K_4=v \cup U_v.$ Observe that $v \in N(u)$ and $v$ is not adjacent to $3$ vertices in $N(u)$ outside $N(v)$ ($u$ has degree $2$ in $v\cup N(v)$ and degree $5$ in $G$). This means $u \cup N(u)$ must also fall under case a) since in case b) any vertex in $G[u \cup N(u)]$ has at most $2$ non-neighbours. So $u \cup N(u)$ makes a triangle $u,v,w$ plus $K_4=u \cup U_u$ (we know the triangles match since only the vertices of the triangle miss $3$ edges in the neighbourhood, which means $v$ is in the triangle and is only adjacent to $u,w$). Analogously we obtain $U_w$ with the same picture. So we find a graph $W$ consisting of $3$ vertex disjoint $K_4$'s each with a singled out vertex such that the singled out vertices make a triangle (see \Cref{fig:3} for an illustration).

\begin{figure}
\begin{minipage}[t]{0.49\textwidth}
    \centering
    \begin{tikzpicture}[xscale=0.8,yscale=0.7]
    
    \defPt{0}{1}{u1}
    \defPt{0}{3}{u2}
    \defPt{-1}{2}{u3}
    \defPt{1}{2}{u4}
    
        \foreach \i in {1,...,4}
        {
        \foreach \j in {2,...,4}
     {
            \draw[] (u\i) -- (u\j);
    }
    }
    \node[] at ($(u1)+(0.25,0)$) {$u$};
    \draw[] (0,2.35) circle [x radius=1.4, y radius=0.9];
    \node[] at (-1.25,3.25) {$U_u$};
    
    \begin{scope}[rotate around={120:(0,0)}]
    \defPt{0}{1}{v1}
    \defPt{0}{3}{v2}
    \defPt{-1}{2}{v3}
    \defPt{1}{2}{v4}
    
    \foreach \i in {1,...,4}
        {
        \foreach \j in {2,...,4}
     {
            \draw[] (v\i) -- (v\j);
    }
    }
    \node[] at ($(v1)+(0.25,0)$) {$v$};
    \draw[] (0,2.35) circle [x radius=1.4, y radius=0.9];
    \node[] at (1.25,3.25) {$U_v$};
    \end{scope}

\begin{scope}[rotate around={240:(0,0)}]
    \defPt{0}{1}{w1}
    \defPt{0}{3}{w2}
    \defPt{-1}{2}{w3}
    \defPt{1}{2}{w4}
    
    \foreach \i in {1,...,4}
    {
        \foreach \j in {2,...,4}
        {
            \draw[] (w\i) -- (w\j);
        }
    }
    \node[] at ($(w1)+(0.25,0)$) {$w$};
    \draw[] (0,2.35) circle [x radius=1.4, y radius=0.9];
    \node[] at (-1.25,3.25) {$U_w$};
    \end{scope}

\draw[] (u1)--(v1)--(w1)--(u1);
\foreach \i in {1,...,4}
{
    \draw[] (u\i) \smvx;
    \draw[] (v\i) \smvx;
    \draw[] (w\i) \smvx;
}
    \end{tikzpicture}
    \caption{Graph $W.$}
    \label{fig:3}
\end{minipage}\hfill
\begin{minipage}[t]{0.49\textwidth}
    \centering
    \begin{tikzpicture}

    \defPt{0}{1}{u1}
    \defPt{0}{3}{u2}
    \defPt{-1}{2}{u3}
    \defPt{1}{2}{u4}
    
        \foreach \i in {1,...,4}
        {
        \foreach \j in {1,...,2}
        {
            \defPt{2*\i}{2*\j}{u\i\j}
            \draw[] (u\i\j) \smvx;
        }
        \draw[] (u\i1) -- (u\i2);
        }
    \draw[] (u11) -- (u21)-- (u31) -- (u41);
    \draw[] (u12) -- (u22)-- (u32) -- (u42);
    \draw[] (u11) -- (u22)-- (u31) -- (u42);
    \draw[] (u12) -- (u21)-- (u32) -- (u41);
    
    \end{tikzpicture}
    \caption{Graph $H.$}
    \label{fig:4}
\end{minipage}
\end{figure}

We know $u,v,w$ have no further edges among vertices of $W$ as they have degree $5$ both in $G$ and $W$. But we claim that also there are no edges between distinct $U_v,U_u,U_w,$ meaning we find $W$ as an induced subgraph. To see this if say $x\in U_v , y\in U_u$ are adjacent, then they can send at most $1$ edge outside $W$ (they have degree $5$ in $G$ and $4$ in $W$). Since $u$ and $x$ are independent and there are at most $8$ vertices adjacent to one of them ($w,U_v,U_u$ and the potential neighbour of $x$ outside $W$) this means that removing $u,N(u),x,N(x)$ removes at most $10$ vertices and leaves us with a $10$ vertex graph which has no independent set of size $3$ (or together with $u$ and $x$ we get a size $5$ one in $G$) so this reduces to the $M(10,3)$ case and we are done. 

Observe that  every vertex of $U_v,U_w,U_u$ has degree $3$ in $W$ and $5$ in $G$ so sends exactly $2$ edges to $G \setminus W$. If there is $x \in U_u$ and $y \in U_w$ which have a common neighbour outside $W,$ then there are at least $5$ vertices outside of $W$ which are not adjacent to either $x$ or $y$. In particular, there is a missing edge among such vertices, which together with $x,y$ and $v$ makes an independent set of size $5$. This means that $N(U_u),N(U_w)$ and $N(U_u)$ when restricted outside of $W$ must be disjoint. In particular, since there are $8$ remaining vertices there is one neighbourhood, say of $U_v$ consisting only of $2$ vertices. Since every vertex of $U_v$ sends $2$ edges outside $W$ this means that $U_v$ and these $2$ vertices span a $K_{3,2}$. Finally, this means that for any vertex $v'$ of $U_v$ we have $v' \cup N(v')$ induces a $K_6$ minus a triangle, which falls under neither of our cases. This is a contradiction since $v'$ has degree $5$ in $G$.
\end{proof}

Let $H$ denote the graph consisting of two vertex disjoint $K_4$'s with two pairs of vertices, one pair per part joined by a $K_{2,2}$. See \Cref{fig:4} for an illustration.

\begin{claim*}
Any vertex of degree $5$ lies in an induced copy of $H$ in $G$ where vertices of degree $3$ in $H$ have degree $6$ in $G$.
\end{claim*}
\begin{proof}
Let $v$ be a vertex of degree $5$, since case b) occurs we obtain two $K_4$'s sharing an edge $uv$. One of the $K_4$'s consists only of vertices of degree $5$ so $u,v$ and call its remaining $2$ vertices $w,s$. So by looking at the neighbourhood of $w$, since case b) must occur we deduce there is a $K_4$ which intersects our two original cliques only in $ws$ (since $v,u$ would have too high degree if they were in the shared edge). This gives us a desired copy of $H$ and it remains to be shown $H$ is induced. 

$u,v,s,w$ all have degree $5$ in $H$ and $5$ in $G$ so they have no additional edges. If we had an edge between $x$ and $y$, both of which are of degree $6$ and which satisfy $x\sim w,y\sim v,$ then $x$ has $4$ edges in $G[H]$ so sends at most $2$ edges outside. Deleting these two edges and our copy of $H$ we deleted at most $10$ vertices among which we deleted $x,N(x),v,N(v)$ and $x$ and $v$ are independent, so we again reduce to the $M(10,3)$ case and are done.   
\end{proof}

\begin{claim*}
No two copies of $H$ as in the previous claim can intersect.
\end{claim*}
\begin{proof}
The central $K_4$ in any such copy of $H$ consists of vertices of degree $5$ in $G$ and in $H$ so they have to be disjoint between copies of $H$. 

If a vertex $x$ of degree $6$ is shared by two copies $H_1,H_2$ of $H,$ then $x$ has either $6$ distinct edges among $H_1 \cup H_2$ or $5$, but then $|H_1 \cap H_2| \ge 2$. Let us delete $H_1,H_2$ and the potential external neighbour of $x$. In either case we deleted at most $15$ vertices. Let us choose $v_i \in H_i$ such that $d(v_i)=5$ and $v_i \nsim x$. This means $x,v_1,v_2$ make an independent set and we deleted them and their neighbourhoods in $G$. This means that among remaining $5$ vertices there can be no missing edges as any such edge together with $x,v_1,v_2$ would span a $K_5$.
\end{proof}

Finally this means that the copies of $H$ we find should be vertex disjoint and since there are $12$ vertices of degree $5$ and each $H$ contains $4$ this means there should be at least $3$ copies so at least $24$ vertices in total giving us a contradiction and completing the proof.
\end{proof}

\section{Upper bounds for the Tur\'an 2-density problem}
\label{appendix}
In this section we show our upper bounds on $M(m,r)$. We begin with some basic observations which will prove useful in computing $2$-densities.
We say two vertices $v$ and $u$ are \textit{equivalent} if there is a transposition swapping vertices $v$ and $u$ in the automorphism group of $G$.

\begin{lem}\label{lem:m2-switching}
Let $H$ be the subgraph of $G$ maximising $d_2(G)$ which has as many vertices as possible. For any pair of equivalent vertices, $H$ contains either both or none of them.
\end{lem}
\begin{proof}
Let $v$ and $u$ be equivalent vertices and assume for the sake of contradiction that $H$ contains $v$ but not $u$. Let $d$ denote the number of neighbours of $v$ in $H$, in particular by equivalence $u$ also has $d$ neighbours in $H \setminus \{v\}$, so in particular at least $d$ neighbours in $H$.

By maximality of $H$ we must have $\frac{e(H)-1}{|H|-2} \ge \frac{e(H)-d-1}{|H|-1-2} \implies d \ge \frac{e(H)-1}{|H|-2}$ or $H \setminus \{v\}$ would have larger $2$-density. Similarly, by comparing with $H \cup \{u\}$ we must have $\frac{e(H)-1}{|H|-2}>\frac{e(H)+d-1}{|H|+1-2} \implies d < \frac{e(H)-1}{|H|-2}$ giving us a contradiction. 
\end{proof}

\begin{lem}\label{thm:ub-even}
For any $k \ge 2$ we have $M(2k,3) \le \frac{k+1}2.$
\end{lem}
\begin{proof}
We take $G$ to be a disjoint union of two $K_k$'s. It is immediate that $\alpha(G) \le 2$ so we only need to show $m_2(G)=m_2(K_k)=\frac{k+1}{2}$. To this end let us take the subgraph $H$ of $G$ as in \Cref{lem:m2-switching}. Since all the vertices within a single clique are equivalent we deduce that $H$ is either $K_k$ or the whole graph. In the first case we are done immediately. In the second case we have $m_2(G)=d_2(G)=\frac{k(k-1)-1}{2k-2}<\frac k2<\frac{k+1}{2}$ giving us a contradiction. 
\end{proof}

\textbf{Remark.} There are more interesting examples, for example one can find tight examples which are $K_k$-free. If $k$ is even one can take $k/2$-th power of the cycle on $2k$ vertices while for odd $k$ one can find a different Cayley graph of the Cyclic group of order $2k$.

\begin{lem}\label{lem:ub-odd-tight}
For any $k \ge 4$ the bound in \Cref{thm:ub-odd} is tight. 
\end{lem}

\begin{proof}
Let $m(t)=\min\left(\frac{t}{k-2},\frac{(k+1)/2-(t-1)^2}{2k-3}\right)$ and let $a$ be the maximiser of this expression (over $1 \le t \le k-2$), as in the proof of \Cref{thm:ub-odd}. Our goal is to find a graph $G$ on $2k-1$ vertices with $\alpha(G)\le 2$ which has $m_2(G) \le \frac{k+1}2-m(a)$. We take $G$ to be the blow-up of $C_5$ with cliques placed inside of parts of size $1,k-1-a,a,a,k-1-a,$ in order around the cycle. Since complement of $G$ is an actual blow-up of $C_5$ it is triangle-free so $\alpha(G) \le 2$.

Observe first that $e(G)=2(k-1-a)+2\binom{k-1}{2}+a^2=k(k-1)+(a-1)^2-1,$ so $$\frac{e(G)-1}{|G|-2}=\frac{k+1}2-\frac{(k+1)/2-(a-1)^2}{2k-3} \le \frac{k+1}2-m(a).$$ 
Note that for $k \ge 4$ we have $\frac{(k+1)/2-(t-1)^2}{2k-3} \le 1/2$ for any $t\ge 1.$ This implies $m(a) \le 1/2$ and in particular since $m_2(K_{k-1}) =k/2 \le (k+1)/2-m(a)$ it is enough to consider only proper subgraphs of size at least $k$ or more when computing $m_2(G)$ (or the desired bound holds). 
Let $H$ be the subgraph of $G$ maximising $d_2(H)$ as in \Cref{lem:m2-switching}. Since all vertices within a part are equivalent we know $H$ must be a union of full parts. Above considerations tell us that $H$ is not the whole graph and has at least $k$ vertices.

Let us also observe that if $\frac{a-1}{k-2}\ge 1/2,$ the second term is smaller (or equal) in the minimum for both $m(a-1)$ and $m(a)$ as well as being decreasing in $t$ which would make $m(a-1)$ larger than $m(a)$, a contradiction. Hence, $a-1 < (k-2)/2,$ implying $a \le (k-1)/2$ and in particular $k-1-a \ge a$. This implies that $\delta(G)=\min(2(k-1-a), k-1) \ge k-1.$ If $H$ has $2k-1-t$ vertices with $1 \le t \le k-2$ we would have 
\begin{align*}
    d_2(H)& =\frac{e(H)-1}{|H|-2} \le \frac{e(G)-t\delta(G)+\binom{t}{2}-1}{2k-t-3} \le \frac{k(k-1)+(a-1)^2-1-t(k-1)+\binom{t}{2}-1}{2k-t-3}\\
    & = \frac{k+1}{2}+\frac{(a-1)^2-(k+1)/2-t(k-3)/2+\binom{t}{2}}{2k-t-3}
    \le \frac{k+1}{2}-\frac{(k+1)/2-(a-1)^2}{2k-3}\le \frac{k+1}{2}-m(a),
\end{align*} 
where in the first inequality we used the fact that $H$ is obtained from $G$ by removing a set of $t$ vertices and hence, these $t$ vertices are incident to at least $\delta(G)t-\binom{t}{2}$ edges.

Finally if $H$ has exactly $k$ vertices, it is easy to see that it misses at least $a$ edges (since it is union of parts of the blow-up and every part except the single vertex one has size at least $a$) which implies $\frac{e(H)-1}{k-2}\le\frac{\binom{k}{2}-a-1}{k-2}= \frac{k+1}2-\frac{a}{k-2} \le \frac{k+1}2-m(a),$ completing the proof.
\end{proof}

The following lemma will be useful in determining when it makes sense to take vertex disjoint unions of graphs in computation of $m_2(G)$. For two graphs $G$ and $H$ we will denote by $G \sqcup H$ the graph obtained by taking a vertex disjoint union of $G$ and $H$.
\begin{lem}
Assume $2e(H)>|H|$ and $2e(G)>|G|$. Then $d_2(G \sqcup H) < \max( d_2(G),d_2(H))$. 
\end{lem}

\begin{proof} Assume otherwise, so $d_2(G \sqcup H) \ge d_2(G),d_2(H)$. We have 
$d_2(G \sqcup H)=\frac{e(G)+e(H)-1}{|G|+|H|-2} \ge \frac{e(G)-1}{|G|-2}=d_2(G) \Leftrightarrow \frac{e(G)+e(H)-1}{e(G)-1} \ge \frac{|G|+|H|-2}{|G|-2}
\Leftrightarrow \frac{e(H)}{e(G)-1} \ge \frac{|H|}{|G|-2}  \Leftrightarrow \frac{e(H)}{|H|} \ge \frac{e(G)-1}{|G|-2}.$ 
Observe now that since $|H|<2e(H)$ we have $\frac{e(H)}{|H|} < \frac{e(H)-1}{|H|-2}=d_2(H)$ so $d_2(G \sqcup H) \ge d_2(G) \implies d_2(H)>d_2(G)$. Repeating with $H$ in place of $G$ we obtain $d_2(G)>d_2(H)$ which is a contradiction.
\end{proof}

The following result shows that bound of \Cref{general-ub} is tight (up to constant factor in the constant term) for the first half of the modulae.

\begin{thm}
Let $k=\ceil{\frac{m}{r-1}}$ and $\ell=m-(k-1)(r-1).$ Provided $\ell \le \frac{r-1}2$ and that $k$ is sufficiently larger than $r$ there exists $c_r>0$ such that $M(m,r) \le \frac{k+1}{2}-\frac{c_r}{\sqrt{k}}.$
\end{thm}
\begin{proof}
Let us first describe our graph $G$. We let $G$ consist of $\ell$ vertex disjoint copies of our graph $G'$ with $2k-1$ vertices from \Cref{lem:ub-odd-tight} and $r-1-2\ell\ge 0$ copies of $K_{k-1}.$ The only information we will use about $G'$ is that $m_2(G') \le \frac{k+1}{2}-\frac{1}{2\sqrt{k+1}}$ and that $\alpha(G') \le 2$. 

Observe first that $\alpha(G)<r$ as any set of $r$ vertices of $G$ must have either $3$ vertices in some copy of $G'$ or $2$ vertices in a copy of $K_{k-1}$, either way there is an edge and the set is not independent.

On the other hand observe that any graph $H$ with $\alpha(H)< r$ and more than $3r$ vertices must satisfy $2e(H)>|H|$, since there can be at most $(1-1/r)|H|^2/2$ non-edges by Tur\'an's theorem so $e(H) \ge |H|^2/(2r)-|H|/2$. 

Let us now take $H$ to be a subgraph of $G$ maximising $d_2(H)$ and our goal is to show that, upon minor modification this $H$ must live inside a single copy of $G'$ or $K_{k-1}$ which will let us complete the proof since we know both have small enough $m_2$. 

We know that $H$ must have at least $k$ vertices or we obtain $m_2(G) =d_2(H)\le m_2(K_{k-1})=k/2$. This means that if $H$ has at most $3r$ vertices inside some copy of $G'$ or $K_{k-1},$ we can remove them and this only changes the $2$-density by $O(\frac1k) \le o(\frac1{\sqrt{k}})$, where we are taking $k$ sufficiently larger than $r$. Let us call the graph $H_1$ obtained from $H$ by deleting all vertices of $H$ belonging to a copy of $G'$ or $K_{k-1}$ if there are at most $3r$ of them inside that copy. Our new graph $H_1$ by above observation has $d_2(H_1) \ge d_2(H)-O(\frac{1}{k})$ and has the property that inside any copy of $G'$ or $K_{k-1}$ it has either none or more than $3r$ vertices. 

Applying the preceding lemma iteratively to restrictions of $H_1$ to a single copy of $G'$ or $K_{k-1}$ we can find a subgraph $H_2$ of either $G'$ or $K_{k-1}$ which satisfies $d_2(H_2) > d_2(H_1)\ge d_2(H)-O(\frac{1}{k})$. Since we know $d_2(H_2) \le \max (m_2(K_{k-1}),m_2(G'))$ and we know $m_2(K_{k-1})=k/2$ and $m_2(G')\le\frac{k+1}{2}-\frac{1}{2\sqrt{k+1}}$ so either way $m_2(G)=d_2(H)$ is small enough.
\end{proof}

\end{document}